\documentclass[12pt]{article}
\usepackage{amsfonts}
\usepackage{stmaryrd}
\usepackage{graphicx}
\usepackage{epsfig}
\usepackage{xcolor}
\usepackage{amsmath}
\usepackage{amsthm}
\usepackage{amssymb}
\usepackage{mathrsfs}
\usepackage{float}
\usepackage{multirow}
\usepackage{verbatim}
\usepackage{fancyhdr}
\usepackage{subfigure}
\usepackage{color}
\usepackage{mathtools}
\usepackage{mathrsfs}
\usepackage{sectsty}
\usepackage[title]{appendix}
\usepackage{threeparttable}
\usepackage{dcolumn}
\usepackage{booktabs}
\usepackage{indentfirst}
\usepackage{setspace}
\usepackage{bm}
\usepackage{enumerate}
\usepackage{lineno}
\usepackage{bbm}
\usepackage[labelfont=bf,labelsep=period]{caption}
\newcolumntype{d}[1]{D{.}{.}{#1}}
\usepackage{hyperref}
\hypersetup{colorlinks,linkcolor=blue,citecolor=blue}

\newtheorem{Thm}{Theorem}[section]
\newtheorem{Ex}{Example}[section]

\newtheorem{Remark}{Remark}[section]

\allowdisplaybreaks

\title{Non-oscillatory entropy stable DG schemes for hyperbolic conservation law}

\begin{document}

\pagenumbering{arabic}
\baselineskip=1.3pc

\vspace*{0.5in}

\begin{center}

{{\bf Non-oscillatory entropy stable DG schemes for hyperbolic conservation law}}

\end{center}

\vspace{.03in}

\centerline{
Yuchang Liu\footnote{School of Mathematical Sciences,
         University of Science and Technology of China,
         Hefei, Anhui 230026, P.R. China.  
         E-mail: lissandra@mail.ustc.edu.cn.},
Wei Guo\footnote{Department of Mathematics and Statistics, Texas Tech University, Lubbock, TX, 70409, USA. 
E-mail: weimath.guo@ttu.edu.
Research supported by NSF grant NSF-DMS-2111383, Air Force Office of Scientific Research FA9550-18-1-0257.
},
Yan Jiang\footnote{School of Mathematical Sciences,
         University of Science and Technology of China, Hefei,
         Anhui 230026, P.R. China.  
         E-mail: jiangy@ustc.edu.cn.
         Research supported by NSFC grant 12271499 and Cyrus Tang Foundation. },
and Mengping Zhang\footnote{School of Mathematical Sciences,
         University of Science and Technology of China, Hefei,
         Anhui 230026, P.R. China.  
         E-mail: mpzhang@ustc.edu.cn.}
}

\vspace{.1in}

\noindent
{\bf Abstract: }In this paper, we propose a class of non-oscillatory, entropy-stable discontinuous Galerkin (NOES-DG) schemes for solving hyperbolic conservation laws. By incorporating a specific form of artificial viscosity, our new scheme directly controls entropy production and suppresses spurious oscillations. 
To address the stiffness introduced by the artificial terms, which can restrict severely time step sizes, we employ the integrating factor strong stability-preserving Runge-Kutta method for time discretization. Furthermore, our method remains compatible with positivity-preserving limiters under suitable CFL conditions in extreme cases. Various numerical examples demonstrate the efficiency of the proposed scheme, showing that it maintains high-order accuracy in smooth regions and avoids spurious oscillations near discontinuities.

\vspace{.1in}

\noindent
\textbf{Key Words:} hyperbolic conservation laws,
discontinuous Galerkin method,
entropy stability, nonoscillatory.

\section{Introduction}

In this paper, we focus on simulating hyperbolic conservation laws, which describe how the balance of conserved quantities within a domain is determined by the flux across its boundaries. A celebrated example is the compressible Euler equations in gas dynamics. 
The discontinuous Galerkin (DG) method has emerged as a powerful and versatile tool for solving these problems, offering several distinctive advantages such as local conservation, high order accuracy, flexibility in handling complex geometries and general boundary conditions, as well as ease of adaptivity and parallel implementation. However, challenges remain in approximating the entropy solution with both numerical and theoretical justifications. The entropy solution is the unique physically relevant weak solution characterized by fulfilling entropy inequalities. Numerical methods that preserve these inequalities at the discrete level are called {\em entropy stable}. Under certain conditions, such schemes can recover the desired entropy solution upon convergence, see \cite{lax1960systems}. Meanwhile, the entropy solution may develop discontinuities or sharp gradient structures in finite time. Even entropy stable schemes may still produce spurious oscillations in the presence of such singular structures. These oscillations can severely contaminate approximation accuracy or lead to simulation failure. Therefore, novel techniques are needed to remove oscillations while preserving the entropy stability property.
The main objective of this work is to develop a simple, computationally efficient, and theoretically justified DG method that can simultaneously achieve discrete entropy stability and avoid nonphysical oscillations.

The past few decades have witnessed the tremendous development of the provably entropy conservative and entropy stable numerical methods. Notably, in the seminal work by Tadmor \cite{tadmor1987numerical,tadmor2003entropy}, the concepts of entropy conserving and entropy stable numerical fluxes were introduced in the context of finite volume schemes. This foundational work has significantly influenced the subsequent developments. For instance, in \cite{lefloch2002fully}, a procedure was proposed for designing high-order entropy conservative fluxes. Building upon these components, an entropy stable and arbitrarily high order version of ENO schemes, termed TeCNO, was developed in \cite{fjordholm2012arbitrarily}. 
Within the framework of DG discretization, it is shown in \cite{jiang1994cell} that the classic semi-discrete DG scheme is entropy stable. However, this property is limited to the square entropy function and requires exact integration.
To ensure entropy stability for a general entropy function and under quadrature approximations, two primary approaches have been developed. The first approach leverages the summation-by-parts (SBP) methodology \cite{fernandez2014generalized, fernandez2014review, svard2014review, del2018simultaneous, hicken2016multidimensional}. Pioneering contributions include \cite{carpenter2014entropy,gassner2013skew,gassner2016well}.
Furthermore, in \cite{chen2017entropy}, a Gauss-Lobatto quadrature-based entropy stable DG scheme is developed by introducing SBP operators in conjunction with the flux differencing technique \cite{fisher2013high, carpenter2014entropy, carpenter2016entropy}. 
In \cite{chan2018discretely, chan2019efficient}, a set of entropy stable DG methods allowing for arbitrary quadrature rules is developed by designing the hybridized SBP operators that combine the volume and surface quadrature nodes. 
The second approach was first introduced in \cite{abgrall2018general}, which seeks to explicitly control the production of entropy within each element by incorporating an artificial term in the DG weak formulation. Unlike the first approach, it does not rely on the flux differencing technique. In \cite{gaburro2023high}, inspired by \cite{abgrall2018general}, a fully discrete entropy conservative DG method under the ADER framework is proposed by choosing a novel artificial diffusion term for entropy correction together with the relaxation approach for time integration \cite{ketcheson2019relaxation}. However, this approach cannot effectively handle problems with discontinuous solutions, as it lacks the ability to control spurious oscillations.

In this work, motivated by \cite{gaburro2023high}, we develop a class of provably (in the semi-discrete sense) high order and entropy stable DG methods that preserve all desired properties of the classic DG method and is capable of eliminating spurious oscillations in the presence of shocks. The difficulty lies in the design of an artificial term that can simultaneously control entropy production and oscillations, as well as preserve the original high order accuracy. Furthermore, it is also desired that the modified formulation is convenient to implement and cost-effective. To address this, we generalize the entropy-based artificial diffusion framework \cite{guermond2011entropy} by incorporating an idea from the essentially oscillation-free DG (OFDG) method  \cite{liu2021oscillation}. The OFDG method employs a projection-based damping term in the DG formulation to control spurious oscillations,  and was extended in \cite{peng2024oedg} to ensure a scale-invariant property. 
In addition, \cite{peng2024oedg}  showed that the damping technique, which damps the modal coefficients of the high-order moments of the numerical solution, is closely related to the artificial viscosity approach.
In particular, the damping coefficient is adaptively adjusted by measuring discontinuity intensity, decreasing in smooth regions and increasing near discontinuities. In this work, we employ an artificial diffusion term rather than a damping term, and furthermore, the coefficient is determined to ensure both the entropy stability and oscillation-free property. Our work introduces three key improvements over existing approaches. First, in \cite{gaburro2023high}, the authors claimed that an arbitrary  piecewise polynomial approximation or interpolation with desired high order accuracy for  the entropy variable can ensure the entropy dissipation property. However, we disprove this claim by providing a counterexample. Moreover, we demonstrate that a specific interpolation property is required for ensuring provable entropy stability. Second, as mentioned in \cite{chen2020review}, the denominator of the coefficient in the entropy correction term can approach zero more rapidly than the numerator, potentially compromising accuracy. To circumvent the difficulty, we propose to further incorporate the oscillation coefficient to control the main coefficient, thereby preserving the accuracy of the scheme. 
Last, the proposed method attempts to directly control entropy production, allowing it to accommodate arbitrary monotone numerical fluxes and quadrature rules. Moreover, the algorithm can be implemented by reusing a classic DG code with only the artificial viscosity term added, significantly reducing development time.

Note that the proposed entropy artificial viscosity term may incur additional stiffness, which can lead to stringent CFL time step restriction when a standard  explicit strong-stability-preserving Runge-Kutta (SSP-RK) time discretization is used. To address this issue, we implement the explicit integral factor SSP-RK time discretization \cite{chen2021krylov}, which can greatly mitigate the time step constraint. Furthermore, for extreme problems such as the high Mach number astrophysical jet problem, it is necessary to ensure positivity of density and pressure to avoid simulation failures. Note that, as the artificial viscosity term is locally conservative, we can further adapt the positivity-preserving (PP) limiter developed in \cite{zhang2010positivity} with an appropriate adjustment on the CFL condition.

The rest of this paper is organized as follows. In Section 2, we introduce the formulation of semi-discrete non-oscillatory entropy stable DG schemes for both one-dimensional and two-dimensional scalar equations and systems, along with theoretical analysis on mass conservation, discrete entropy stability, and error estimation. In Section 3, we present the fully discrete formulations, integrating an integral factor SSP-RK time discretization and a modified PP limiter. In Section 4, various numerical examples are provided to demonstrate the efficiency and efficacy of our proposed scheme. The conclusions are discussed in Section 5.

\section{Semi-discrete entropy stable DG schemes}
\subsection{1D scalar cases}

Consider the following one-dimensional scalar conservation law 
\begin{equation}\label{hcl}
\begin{cases}
    u_t+f\left( u \right) _x=0,\quad x\in \Omega=[a,b], \, t>0,\\
    u\left( x,0 \right) =u_0\left( x \right) .\\
\end{cases}
\end{equation}
For simplicity, we assume that the exact solution has either periodic or compact support boundary conditions on domain $\Omega$.

Let $(U,F)$ denote an entropy pair, satisfying the relation $F'(u)=U'(u)f'(u)$.
Denote the entropy variable by $v = U'(u)$, and define $A(u)=U''(u)^{-1}$ which is always non-negative.
Then a weak solution of \eqref{hcl} is called an {\em entropy solution} if for all entropy pairs, we have that
\begin{equation}\label{es} 
    U(u)_t + F(u)_x \le 0 
\end{equation}
in the weak form.
Furthermore, integrating the entropy condition \eqref{es} in space, we get
\begin{equation} \label{eq:es-2}
    \frac{\mathrm d}{\mathrm dt}\int_{\Omega} U(u)\mathrm dx\le 0. 
\end{equation}
This means that the total entropy is non-increasing with respect to time.
In the following, we will design a DG scheme to satisfy the inequality \eqref{es} or \eqref{eq:es-2} numerically.

\subsubsection{The classical DG discretization}

We start with the classical DG scheme for \eqref{hcl}. 
Suppose the domain is divided into cells
$$ a= x_{1/2}<x_{3/2}<\cdots<x_{N+1/2} =b ,
\quad h_i=x_{i+1/2}-x_{i-1/2}.$$
Let $K_i=[x_{i-1/2},x_{i+1/2}]$ be the cell with its center $x_i=(x_{i+1/2}+x_{i-1/2})/2$, and 
$\mathcal K=\{K_i\}$ be the collocation of all cells. Without loss of generality, we assume the mesh is uniform and denote by $h=h_i$. 

Define the finite element space as
\begin{equation}
    V_h^k = \{w(x): w(x)|_{K}\in P^k(K),\ \forall K\in\mathcal K  \}, 
\end{equation}
where $P^k(K)$ is the space of polynomials of degree at most $k$ on cell $K$. Notice that for any $w\in V_h^k$, it can be discontinuous at the cell boundary $x_{i\pm 1/2}$. 
In the following, we let $w^+_{i+1/2}$ (resp. $w^-_{i+1/2}$) denote the limiting value of 
$w$ at $x_{i+1/2}$ from the element $K_{i+1}$ (resp. $K_i$), 
$[\![w]\!]_{i+1/2}=w^{+}_{i+1/2}-w^{-}_{i+1/2}$ denote its jump, 
and $\{w\}_{i+1/2} =\frac{1}{2} (w^{+}_{i+1/2} +w^{-}_{i+1/2})$ be its average at $x_{i+1/2}$. 

The semi-discrete DG scheme is defined as: seek $u_h\in V_h^k$, such that for all $w\in V_h^k$ and  $K_i\in \mathcal{K}$,
\begin{equation}\label{eq:DG_standard}
    \int_{K_i}{\frac{\partial u_h}{\partial t} w \,\mathrm{d}x}
    =\int_{K_i}{f\left( u_h \right) \frac{\partial w}{\partial x}\mathrm{d}x}
    -\hat{f}_{i+1/2}w _{i+1/2}^{-}+\hat{f}_{i-1/2}w _{i-1/2}^{+}.
\end{equation}
Here $\hat f_{i+1/2} = \hat f (u_h|^-_{i+1/2}, u_h|^+_{i+1/2})$ is the monotone numerical flux at $x_{i+1/2}$. In this work, we use the local Lax-Friedrichs flux
$$
    \hat f(u^-,u^+) = \frac{1}{2}\left[ f(u^+)+f(u^-) - \alpha(u^+ - u^-) \right],\quad 
    \alpha=\max\limits_{u\in I(u^-,u^+)} \left|f'(u)\right|, 
$$
where $I(u^-,u^+)$ is the interval defined between $u^-$ and $u^+$.

\subsubsection{The entropy stable non-oscillatory DG schemes}

Considering the equation \eqref{hcl} with an additional local viscosity in cell $K_i$
\begin{equation}\label{vis2} 
    u_t + f(u)_x = \sigma_i \left( \nu_i(x)A(u)v(u)_x \right)_x,\quad x\in K_i,
\end{equation}
where $\sigma_i$ is a local non-negative parameter, and $\nu_i$ satisfies
\begin{equation} \label{eq:cond_nu}
    \nu_i(x_{i- 1/2}) = \nu_i(x_{i+1/2}) = 0 \quad \text{and} 
    \quad \nu_i(x) \geq 0,\  x\in K_i. 
\end{equation}
In this work, we simply let
\begin{equation}\label{eq:nu}
    \nu_i(x) = 1 - X_i^2(x),\quad X_i(x) = \frac{x - x_i}{h/2}.
\end{equation}
Notice that the added viscous term is similar to that introduced in \cite{wei2024jump}. However,  \cite{wei2024jump} focuses solely on eliminating spurious oscillation by adding viscosity based on the smoothness of the solution $u$ itself. In contrast, we are more concerned with controlling entropy. The added viscosity term relies on the entropy terms $A(u)$ and $v(u)$ rather than $u$ itself, which provides additional entropy dissipation. By carefully choosing the parameter $\sigma_i$ we are able to achieve numerical entropy stability while avoiding spurious oscillations.

The semi-discrete non-oscillatory entropy stable DG scheme (denoted as NOES-DG) is defined based on the modified viscous equation \eqref{vis2}: 
find $u_h\in V_h^k$, such that for all $w\in V_h^k$ and $K_i\in\mathcal K$,
\begin{equation}\label{scheme1}
\begin{aligned}
    \int_{K_i}{\frac{\partial u_h}{\partial t}w\, \mathrm{d}x} 
    =& \int_{K_i}{f\left( u_h \right) \frac{\partial w}{\partial x}\mathrm{d}x}-\hat{f}_{i+1/2}w _{i+1/2}^{-}+\hat{f}_{i-1/2}w _{i-1/2}^{+} \\
    -&\sigma _i\int_{K_i}{\nu _i \frac{\partial v_h}{\partial x} A\left( u_h \right) \frac{\partial w}{\partial x} \mathrm{d}x}.
\end{aligned}
\end{equation}
Note that $\nu_i(x)$ vanishes at $x_{i\pm 1/2}$ \eqref{eq:cond_nu}, and hence the boundary terms resulting from integration by parts are nullified, simplifying the numerical formulation.
The function $v_h\in V_h^k$ is a piecewise polynomial approximating the entropy variable $v(u_h)$ with condition
\begin{equation}\label{vh} 
    v_h(x_{i+1/2}^-) = v(u_h(x_{i+1/2}^-)),\quad 
    v_h(x_{i-1/2}^+) = v(u_h(x_{i-1/2}^+)).
\end{equation}
This condition is critical for entropy control, which will be discussed in detail below. We let $v_h$ be the $(k+1)$-th order interpolating polynomial of $v(u_h)$ based on the Gauss-Lobatto quadrature points in $K_i$, hence automatically satisfying \eqref{vh}.

In addition, the parameter $\sigma_i$ is defined as 
\begin{equation}\label{eq:sigma}
\sigma_i = \max\left\{\sigma_i^{jump},\sigma_i^{entropy}\right\},
\end{equation}
where
\begin{subequations}
\begin{align}
\sigma _{i}^{jump}&=c_f\left( h\left\| [\![ u_h ]\!] \right\| _{\partial K_i}+\sum_{l=1}^k{l\left( l+1 \right) h^{l+1}\left\| [\![ \partial _{x}^{l}u_h ]\!] \right\| _{\partial K_i}} \right),\label{eq:jump}
\\ \sigma _{i}^{entropy}&=\max\left\{\frac{F_i}{E_i},0\right\}, \label{eq:sigma2}
\\E_i &= \int_{K_i}\nu_i\frac{\partial v_h}{\partial x}A\frac{\partial v_h}{\partial x}\mathrm dx,\label{eq:E}
\\ F_i &= \displaystyle\hat{F}_{i+1/2}-\hat{F}_{i-1/2}-\hat{f}_{i+1/2}^{C}v_{i+1/2}^{-}+\hat{f}_{i-1/2}^{C}v_{i-1/2}^{+}+\int_{K_i}{f\left( u_h \right)  \frac{\partial v_h}{\partial x}\mathrm{d}x}\label{eq:F}.
\end{align}
\end{subequations}
Moreover, to ensure entropy stability, we need that the quadrature rules used to approximate the integrals in \eqref{scheme1} are identical for approximating the integrals in $E_i$ and $F_i$.
The choice of  $\sigma_i^{entropy}$ is proposed in \cite{gaburro2023high} to balance the entropy production.  
Here, $\hat F$ in \eqref{eq:F} denotes a numerical entropy flux consistent with $F$, i.e., $\hat F(u,u)=F(u)$. In this work, we choose $\hat F$ to be the central flux
\begin{equation}\label{fhat}
\hat F(u^-,u^+) = \frac{1}{2}\left( F(u^+) + F(u^-) \right). 
\end{equation}

\noindent
Assume that the numerical flux is split into $\hat f=\hat f^C + \hat f^D$, where $\hat f^C$ and $\hat f^D$ denote the central part and the diffusive part, respectively. For example, for the Lax-Friedrichs flux, we have
\begin{equation} \label{LFCD}
\hat f^C(u^-,u^+) = \frac{1}{2}(f(u^+) + f(u^-)),\quad 
\hat f^D(u^-,u^+) = -\frac{1}{2}\alpha(u^+-u^-). 
\end{equation}
Note that several other popular monotone numerical fluxes, such as the Engquist–Osher flux and Godunov flux, allow for such a splitting \cite{leveque2002finite}. 
Below, we will show that when $\sigma_i\ge\sigma_i^{entropy}$, the scheme is entropy stable.

Meanwhile, the DG formulation \eqref{scheme1} with $\sigma_i = \sigma^{entropy}_i$, while entropy stable, may still generate spurious oscillations in the presence of strong shocks. To remedy such a drawback, we further incorporate $\sigma_i^{jump}$ to suppress oscillations and define $\sigma_i$ as given in \eqref{eq:sigma}. $\sigma _{i}^{jump}$, originally proposed by the OFDG method \cite{liu2021oscillation} and slightly modified in \cite{wei2024jump}, measures discontinuity intensity and is defined with 
$$\left\|[\![\partial_x^lu_h]\!]\right\|_{\partial K_i} = \left|[\![\frac{\partial^l u_h}{\partial x^l}]\!]_{x_{i-1/2}}\right| + \left|[\![\frac{\partial^l u_h}{\partial x^l}]\!]_{x_{i+1/2}}\right|.$$
The coefficient $c_f$  is defined as $c_f = c_0\max\limits_{ x\in K_i}\{ \left|f'(u_h(x)) \right| \}$, as introduced in \cite{peng2024oedg}, ensuring a scale-invariant property. Here, $c_0$ is a constant. However, since $\sigma_{i}^{entropy}$ in \eqref{eq:sigma2} is not scale-invariant, the proposed technique using \eqref{eq:sigma} does not possess such a property.

Moreover,  $E_i$ may approach zero, causing the parameter  $\sigma^{entropy}_i$ in \eqref{eq:sigma2} to become very large. This can greatly compromise the accuracy of the scheme \cite{chen2020review} and incur significant stiffness.
Note that even for smooth problems, we do not have  control over the parameter $\sigma^{entropy}_i$.
As an example, consider the linear flux $f(u) = u$ with the square entropy $U(u)=u^2$, and $u_h$ is a $(k+1)$-th order approximate of $u$, then we have the estimates of $F_i$ and $E_i$
$$ F_i=\frac{1}{4}\left[(u_h|_{i+1/2}^+-u_h|_{i+1/2}^-)^2-(u_h|_{i-1/2}^+-u_h|_{i-1/2}^-)^2\right]\lesssim h^{2k+2},
\quad E_i\lesssim h^2.$$
Here, $A \lesssim B$ means that there exists a constant $c > 0$ independent of $h$ such that $A \leq c B$.
However, we cannot directly estimate the ratio $F_i/E_i$ unless there is a strictly positive lower bound of $E_i$,  which is not the case in practice. 
Therefore, in the simulation, we set
\begin{equation}
    \sigma _{i}^{entropy}=\min \left\{ \max \left\{\frac{F_i}{E_i},0\right\},C\, \sigma_i^{jump}\right\},\label{sigma2}
\end{equation}
instead of \eqref{eq:sigma2} to overcome the difficulty,  where
$C > 1$ is a free parameter.

\subsubsection{Theoretical properties}

In this section, we analyze the proposed NOES-DG scheme \eqref{scheme1} and establish its several theoretical properties, including mass conservation, discrete entropy stability, and error estimates. 

\begin{Thm} \textbf{(Mass Conservation)}
The NOES-DG scheme \eqref{scheme1} is conservative, i.e.
\begin{equation}\label{eq:conservation}
\frac{\mathrm d }{\mathrm d t}\int_{\Omega}u_h\mathrm{d}x =0.
\end{equation}
\end{Thm}

\begin{proof}
Taking $w = 1$ in \eqref{scheme1}, we get
$$
\int_{K_i}{\frac{\partial u_h}{\partial t}\mathrm{d}x}+\hat{f}_{i+1/2}-\hat{f}_{i-1/2}=0.
$$
Summing over $i$ yields the conservation \eqref{eq:conservation}.
\end{proof}

\vspace{0.3cm}
\begin{Thm}\label{esthm}  \textbf{(Entropy stability)}
If 
\begin{equation}\label{sigma3}
\max \left\{\frac{F_i}{E_i},0\right\}<C\,\sigma_i^{jump},\quad\forall K_i\in\mathcal K,
\end{equation}
then the NOES-DG scheme \eqref{scheme1} is entropy stable in the sense of
\begin{equation}\label{es0}
\int_{\Omega}{\frac{\partial u_h}{\partial t}v_h\mathrm{d}x} \le 0.
\end{equation}
\end{Thm}

\begin{proof}
Taking $w = v_h$ in scheme \eqref{scheme1}, we get
$$
\int_{K_i}\frac{\partial u_h}{\partial t}v_h\mathrm{d}x\le -\hat{F}_{i+1/2}+\hat{F}_{i-1/2}-\hat{f}_{i+1/2}^{D}v_h|_{i+1/2}^{-}+\hat{f}_{i-1/2}^{D}v_h|_{i-1/2}^{+}.
$$
Sum it over $i$ and we can obtain 
\begin{equation}\label{es1}
\int_{\Omega}{\frac{\partial u_h}{\partial t}v_h\mathrm{d}x}\le -\sum_i{\left( \hat{f}_{i+1/2}^{D}v_h|_{i+1/2}^{-}-\hat{f}_{i-1/2}^{D}v_h|_{i-1/2}^{+} \right)}.
\end{equation}
Taking advantage of the periodic or compact support boundary, we get
\begin{align}
\mathrm{RHS}&=-\sum_i{\left( \hat{f}_{i+1/2}^{D}v_h|_{i+1/2}^{-}-\hat{f}_{i+1/2}^{D}v_h|_{i+1/2}^{+} \right)}
\nonumber\\
&=-\sum_i{\frac{1}{2}\alpha _{i+1/2}\left( u_h|_{i+1/2}^{+}-u_h|_{i+1/2}^{-} \right) \left( v_h|_{i+1/2}^{+}-v_h|_{i+1/2}^{-} \right)}
\nonumber\\
&=-\sum_i{\frac{1}{2}\alpha _{i+1/2}\left( u_h|_{i+1/2}^{+}-u_h|_{i+1/2}^{-} \right) \left( v\left( u_h|_{i+1/2}^{+} \right) -v\left( u_h|_{i+1/2}^{-} \right) \right)}
\label{es2}\\
&=-\sum_i{\frac{1}{2}\alpha _{i+1/2}\left( u_h|_{i+1/2}^{+}-u_h|_{i+1/2}^{-} \right) ^2\frac{v\left( u_h|_{i+1/2}^{+} \right) -v\left( u_h|_{i+1/2}^{-} \right)}{u_h|_{i+1/2}^{+}-u_h|_{i+1/2}^{-}}}
\nonumber\\
&=-\sum_i{\frac{1}{2}\alpha _{i+1/2}\left( u_h|_{i+1/2}^{+}-u_h|_{i+1/2}^{-} \right) ^2U''\left( \xi \right)} \le 0,  \nonumber
\end{align}
where $\xi\in I(u_h|^-_{i+1/2},u_h|^+_{i+1/2})$. Then, we can have \eqref{es0}.
\end{proof}

\begin{Remark}\label{rmk1}
In this proof, we observe the necessity of $v_h$ to satisfy the equation \eqref{vh}, as we must have 
$v(u_h|_{i\pm 1/2}^\mp)=v_h|^\mp_{i\pm 1/2}$ in the third equation of \eqref{es2}.
In \cite{gaburro2023high}, the authors claimed that $v_h\in V_h^k$ can be an arbitrary piecewise polynomial approximation of $v(u_h)$,
because they believed that the following inequality always holds
\begin{equation}\label{esfd}
-\sum_i{\left( \hat{f}_{i+1/2}^{D}v_h|_{i+1/2}^{-} 
-\hat{f}_{i-1/2}^{D}v_h|_{i-1/2}^{+} \right)}\le 0.
\end{equation} 
Thus, \eqref{es1} directly leads to \eqref{es0}. 
However, we have found that \eqref{esfd} is not always valid, as demonstrated by the following counterexample. 
Assume $\Omega=[-1,2]$ is divided by a single element $K=\Omega$, and consider the entropy function $U(u)=-\ln u$. 
The numerical solution is a $P^2$ polynomial $u_h(x)=x^2+1$, extended periodically. Consequently, we have $ v\left( u_h \right) =-1/(x^2+1) $. If we define $v_h$ as the interpolation of $v(u_h)$ at $x=-1^+, x=-0.9$ and $x=-0.8$, then we obtain that 
$v_h|_{1/2}^-=v_h(2^-)<v_h(-1^+) =v_h|_{-1/2}^+ =v_h|_{1/2}^+$. 
Moreover, it is obvious that $u_h|_{1/2}^- = u_h(2^-)> u_h(-1^+) =u_h|_{-1/2}^+ =u_h|_{1/2}^+$. Now we have
$$
-\left( \hat{f}_{1/2}^{D}v_h|_{1/2}^{-}-\hat{f}_{-1/2}^{D}v_h|_{-1/2}^{+} \right) 
=-\alpha_{1/2} \left( u_h|_{1/2}^{+}-u_h|_{1/2}^{-} \right) \left( v_h|_{1/2}^{+}-v_h|_{1/2}^{-} \right) >0,
$$
which contradicts \eqref{esfd}. This indicates that $v_h$ cannot be chosen as an arbitrary piecewise polynomial approximation or interpolation of $v(u_h)$. However, with the condition \eqref{vh}, we are able to show that \eqref{esfd} holds. 
\end{Remark}

\begin{Remark}\label{rmk2}
Although $u_x = Av_x$ holds for the exact solution, we must use the form $A(u_h)\dfrac{\partial v_h}{\partial x}$ in the last term of \eqref{scheme1} instead of $\dfrac{\partial u_h}{\partial x}$. This guarantees that, when taking $w = v_h$ in \eqref{scheme1}, the integral in the last term could be non-negative, allowing us to finally arrive at \eqref{es0}.
\end{Remark}

\begin{Remark}\label{rmk4}
For the general entropy functions, we have that 
\begin{equation*}
\begin{aligned}
    \int_{\Omega}{\frac{\partial u_h}{\partial t}v_h\mathrm{d}x} 
    =& \sum_{i}\int_{K_i}{\frac{\partial u_h}{\partial t}v_h\mathrm{d}x} \\
    =& \sum_{i} \left( \sum_{\alpha=1}^{k+1} \hat\omega_i^\alpha \frac{\partial u_h}{\partial t}(\hat{x}^\alpha_i) v_h(\hat{x}^\alpha_i)  + \mathcal{O}(h^{2k+1}) \right) \\
    =& \sum_{i} \left( \sum_{\alpha=1}^{k+1} \hat\omega_i^\alpha \frac{\partial U(u_h)}{\partial t}(\hat{x}^\alpha_i)  + \mathcal{O}(h^{2k+1}) \right) \\
    =& \frac{\mathrm d}{\mathrm dt}\int_{\Omega} U(u_h) \mathrm{d}x + \mathcal{O}(h^{2k}).
\end{aligned}
\end{equation*}
Here, $\hat\omega_i^\alpha$ and $\hat{x}^\alpha_i$ are the Gauss-Lobatto weights and quadrature points in each cell $K_i$, with $\hat{x}_i^1=x_{i-1/2}$ and $\hat{x}^{k+1}_{i}=x_{i+1/2}$. This indicates that the entropy can be considered stable, with only a small, high order growth. Also see \cite{gaburro2023high}.
\end{Remark}

Next, we present the prior error estimate for the linear scalar conservation law with $f(u) = u$. 
Let us recall some useful inequalities. For a piecewise polynomial $w\in V_h^k$, we have the inverse inequality
\begin{equation}\label{inv}
\left\| \frac{\partial w}{\partial x} \right\|_{L^2(K_i)}\lesssim h^{-1} \left\|w\right\|_{L^2(K_i)},\ \ 
\left\| w\right\|_{L^\infty(K_i)}\lesssim h^{-1/2} \left\|w\right\|_{L^2(K_i)},
\quad \forall K_i\in\mathcal K.  
\end{equation}
And for a given function $u$, its Gauss-Radau projection $P^-u\in V_h^k$ is uniquely defined by
\begin{equation} 
\begin{cases}\int_{K_i}(P^-u-u)w\,\mathrm dx=0, \quad \forall \ w\in P^{k-1}(K_i),\\ 
P^-u(x_{i+1/2}^-) = u(x_{i+1/2}^-). 
\end{cases}
\end{equation}
Assume $u\in H^{k+1}(\Omega)$, then the Gauss-Radau projection $P^-u$ satisfies the error estimate \cite{cockburn2001superconvergence}
\begin{equation}\label{GR}
\left| P^-u-u \right|_{m,K_i}\lesssim h^{k+1-m}\left\| u \right\| _{k+1,K_i}
\end{equation}
for all $K_i\in\mathcal K$ and $0\le m\le k$. Here, $\left|\cdot \right|_{m,K}$ and $\left\| \cdot \right\|_{m,K}$ denote the Sobolev (semi-)norms
$$
\left| u \right|_{m,K}=\left( \int_K{\left( \frac{\partial ^mu}{\partial x^m} \right) ^2\mathrm{d}x} \right) ^{1/2},\quad  \left\| u \right\| _{m,K}=\left( \sum_{l=0}^m{\int_K{\left( \frac{\partial ^lu}{\partial x^l} \right) ^2\mathrm{d}x}} \right) ^{1/2}.
$$
Now we give an error estimate of our scheme.

\begin{Thm}\label{acc} \textbf{(Error Estimate)}
Suppose $u$ is a smooth solution of \eqref{hcl} with periodic or compactly supported boundary conditions. Consider equation \eqref{hcl} with linear flux $f(u) = u$ and square entropy $U=u^2/2$. If the initial data is chosen by the standard $L^2$ projection of $u_0(x)$, i.e.,
$$ \int_{K_i}u_h(x,0)\phi(x)\mathrm dx=\int_{K_i}u_0(x)\phi(x)\mathrm dx,\quad \forall K_i\in\mathcal K,\phi\in V_h^k, $$
then the numerical solution $u_h$ of NOES-DG scheme \eqref{scheme1} with the fluxes \eqref{fhat}-\eqref{LFCD} satisfies
\begin{equation}\label{esti0}
   \left\| u(\cdot,t)-u_h(\cdot,t) \right\|_{L^2(\Omega)} \lesssim h^{k+1}. 
\end{equation} 
\end{Thm}

\begin{proof}
Introduce a short-hand notation
$$
B_i\left( u_h,w \right) =\int_{K_i}{\frac{\partial u_h}{\partial t}w\, \mathrm{d}x}-\int_{K_i}{u_h\frac{\partial w}{\partial x}\mathrm{d}x}
+u_{h}|_{i+1/2}^{-} w_{i+1/2}^{-}-u_{h}|_{i-1/2}^{-} w_{i-1/2}^{+},
$$
then the scheme \eqref{scheme1} can be written as
$$
B_i\left( u_h,w \right) =-\sigma_i\int_{K_i}{\nu _i\frac{\partial u_h}{\partial x}\frac{\partial w}{\partial x}\mathrm{d}x},\quad \forall w\in V_h^k,\ K_i\in\mathcal K
$$
On the other hand, the exact solution $u$ satisfies
$$ B_i(u,w) = 0,\quad \forall w\in V_h^k,\ K_i\in\mathcal K. $$
Hence,
\begin{equation}\label{error} 
B_i(u-u_h,w)=\sigma_i\int_{K_i}\nu_i\frac{\partial u_h}{\partial x}\frac{\partial w}{\partial x}\mathrm{d}x,
\quad\forall w\in V_h^k,\ K_i\in\mathcal K.
\end{equation}
Denote
$$\varepsilon_h=u-P^-u,\quad e_h=P^-u-u_h\in V^k_h.$$ 
Taking $w=e_h$, \eqref{error} becomes
\begin{equation}\label{acc2}
B_i\left( e_h,e_h \right)+\sigma_i\int_{K_i}\nu_i\left(\frac{\partial e_h}{\partial x}\right)^2\mathrm dx =-B_i\left( \varepsilon _h,e_h \right) +\sigma _i\int_{K_i}{\nu _i\frac{\partial (P^-u)}{\partial x}\frac{\partial e_h}{\partial x}\mathrm{d}x}.
\end{equation}
Since $\sigma_i\geq0$, utilizing the periodic or compactly supported boundary conditions, we have that  
\begin{equation*}\label{accnew1}
\begin{aligned}
&\sum_i \left( B_i\left( e_h,e_h \right)+\sigma_i\int_{K_i}\nu_i\left(\frac{\partial e_h}{\partial x}\right)^2\mathrm dx \right) 
\\
&\ge \sum_i \left( \int_{K_i}{\frac{\partial e_h}{\partial t}e_h\mathrm{d}x}-\int_{K_i}{e_h\frac{\partial e_h}{\partial x}\mathrm{d}x}+\left( e_h|_{i+1/2}^{-} \right) ^2-\left( e_h|_{i-1/2}^{-} \right) \left( e_h|_{i-1/2}^{+} \right) \right) 
\\
&=\frac{1}{2}\frac{\mathrm{d}}{\mathrm{d}t}\int_{\Omega} {\left( e_h \right) ^2\mathrm{d}x} +\sum_{i} \left( -\frac{1}{2}\left( e_h|_{i+1/2}^{-} \right) ^2+\frac{1}{2}\left( e_h|_{i-1/2}^{+} \right) ^2+\left( e_h|_{i+1/2}^{-} \right) ^2 -\left( e_h|_{i-1/2}^{-} \right) \left( e_h|_{i-1/2}^{+} \right) \right) 
\\
&=\frac{1}{2}\frac{\mathrm{d}}{\mathrm{d}t}\int_{\Omega} {\left( e_h \right) ^2\mathrm{d}x} + \sum_i \left( \frac{1}{2}\left( e_h|_{i-1/2}^{-} \right) ^2 + \frac{1}{2}\left( e_h|_{i-1/2}^{+} \right) ^2 -\left( e_h|_{i-1/2}^{-} \right) \left( e_h|_{i-1/2}^{+} \right) \right)
\\
&=\frac{1}{2}\frac{\mathrm{d}}{\mathrm{d}t}\int_{\Omega} {\left( e_h \right) ^2\mathrm{d}x} + \sum_i \frac{1}{2}\left( e_h|_{i-1/2}^{-} - e_h|_{i-1/2}^{+} \right) ^2 
\\
&\ge \frac{1}{2}\frac{\mathrm{d}}{\mathrm{d}t}\int_{\Omega} {\left( e_h \right) ^2\mathrm{d}x} .
\end{aligned}
\end{equation*}
On the other hand, using the property of Gauss-Radau projection, we have that
\begin{align*}
B_i&\left( \varepsilon _h,e_h \right) 
\\&=\int_{K_i}{\frac{\partial \varepsilon _h}{\partial t}e_h\mathrm{d}x}-\int_{K_i}{\varepsilon _h\frac{\partial e_h}{\partial x}\mathrm{d}x}+\left( \varepsilon _h|_{i+1/2}^{-} \right) \left( e_h|_{i+1/2}^{-} \right) -\left( \varepsilon _h|_{i-1/2}^{-} \right) \left( e_h|_{i-1/2}^{+} \right) 
\\
&=\int_{K_i}{\frac{\partial \varepsilon _h}{\partial t}e_h\mathrm{d}x}.
\end{align*}
Therefore,
$$
\frac{1}{2}\frac{\mathrm{d}}{\mathrm{d}t}\int_{\Omega}{\left( e_h \right) ^2\mathrm{d}x}\le -\int_{\Omega}{\frac{\partial \varepsilon _h}{\partial t}e_h\mathrm{d}x}+\sum_i{\sigma _i\int_{K_i}{\nu _i\frac{\partial (P^-u)}{\partial x}\frac{\partial e_h}{\partial x}\mathrm{d}x}}.
$$
For the last term, we have
\begin{equation*}
    \label{esti}\begin{aligned}
\sum_i{\sigma _i\int_{K_i}{\nu _i\frac{\partial (P^-u)}{\partial x}\frac{\partial e_h}{\partial x}\mathrm{d}x}}
&\le \sum_i{\sigma _i}\left\| \frac{\partial (P^-u)}{\partial x} \right\| _{L^2\left( K_i \right)}\left\| \frac{\partial e_h}{\partial x} \right\| _{L^2\left( K_i \right)}
\\&
\lesssim \sum_i{\sigma _i} h^{-1/2}\left\| e_h \right\| _{L^2\left( K_i \right)}
\\&
\lesssim \left( \sum_i{\left( \sigma _i \right) ^2}h^{-1} \right) ^{1/2}\left\| e_h \right\| _{L^2\left( \Omega \right)}.\end{aligned}
\end{equation*}
In the second inequality, we have used the error estimate 
\begin{equation*}\label{estiGR}\begin{aligned}
\left\| \frac{\partial \left( P^-u \right)}{\partial x} \right\| _{L^2\left( K_i \right)}&=\left| P^-u \right|_{1,K_i}
\leq \left| u \right|_{1,K_i}+\left| u-P^-u \right|_{1,K_i}
\\
&\lesssim \left| u \right|_{1,K_i}+h^k\left\| u \right\| _{k+1, K_i }\lesssim h^{1/2}.
\end{aligned}
\end{equation*}
For $\sigma_i$, since $u$ is smooth, we have $\left|[\![\partial_x^lu_h]\!]\right| = \left|[\![\partial_x^l(u-u_h)]\!]\right|$. 
According to \eqref{sigma2}, the viscosity coefficient has an upper bound $\sigma_i\le C\sigma_i^{jump}$. 
Hence,
$$\begin{aligned}
\sum_i{\left( \sigma _i \right) ^2}h^{-1}
&\lesssim h\sum_i{\sum_{l=0}^k{\left\| [\![ \partial _{x}^{l}u_h ]\!] \right\|_{\partial K_i} ^2h^{2l}}}
\\
&=h\sum_i{\sum_{l=0}^k{\left\| [\![ \partial _{x}^{l}\left( u-u_h \right) ]\!] \right\|_{\partial K_i} ^2h^{2l}}}
\\
&\lesssim \sum_i \sum_{l=0}^k{\left\| \partial _{x}^{l}e_h \right\| _{L^2\left( K_i \right)}^{2}h^{2l}}
+\sum_i \sum_{l=0}^k{\left\| \partial _{x}^{l}\varepsilon _h \right\| _{L^2\left( K_i \right)}^{2}h^{2l}}  
\\
&\lesssim \left(\left\| e_h \right\| _{L^2\left( \Omega \right)}^{2}+ h^{2k+2}\right). 
\end{aligned}
$$
Then \eqref{acc2} yields
$$
\frac{1}{2}\frac{\mathrm{d}}{\mathrm{d}t}\int_{\Omega}{\left( e_h \right) ^2\mathrm{d}x}\lesssim \left( \int_{\Omega}{\left( e_h \right) ^2\mathrm{d}x}+h^{2k+2} \right) 
$$
Using the Gronwall's inequality, we can get that
$$
\left\| e_h\left( \cdot ,t\right) \right\|_{L^2(\Omega)} 
\lesssim \left\| e_h\left( \cdot ,0 \right) \right\|_{L^2(\Omega)} +h^{k+1}.
$$
Since $u_h(\cdot,0)$ is obtained by the standard $L^2$ projection of $u$, resulting $\left\|(u-u_h)(\cdot,0)\right\|\lesssim h^{k+1}$, thus $\left\|e_h(\cdot,0)\right\|\lesssim h^{k+1}$. Finally, we can obtain that
$$
 \left\| u(\cdot,t)-u_h(\cdot,t) \right\|_{L^2(\Omega)}\lesssim h^{k+1}.$$
\end{proof}

\begin{Remark}

For arbitrary entropy function $U(u)$, if we assume $\partial _x^lu_h\lesssim 1,\,l=0,1,\cdots,k$, 
and $A^{(k)}(u)$ is continuous with $u$, then the error estimate \eqref{esti0} still holds. We give a brief proof of this case under the assumptions. 
For a general $U(u)$, \eqref{error} is changed to
$$
B_i(u-u_h,w )=\sigma _i\int_{K_i}{\nu _i} \frac{\partial v_h}{\partial x}A\left( u_h \right) \frac{\partial w}{\partial x}\mathrm{d}x, \quad \forall w \in V_{h}^{k}, K_i\in \mathcal{K} ,
$$
deriving that
$$
B_i\left( e_h,e_h \right) =-B_i\left( \varepsilon _h,e_h \right) +\sigma _i\int_{K_i}{\nu _i\frac{\partial v_h}{\partial x}A\left( u_h \right) \frac{\partial e_h}{\partial x}\mathrm{d}x}.
$$
With the help of monotone numerical fluxes, we still have 
$\sum_i B_i\left( e_h,e_h \right) \ge \frac{1}{2}\frac{\mathrm{d}}{\mathrm{d}t}\int_{\Omega} {\left( e_h \right) ^2\mathrm{d}x}$.
Consequently, 
$$
\frac{1}{2}\frac{\mathrm{d}}{\mathrm{d}t}\int_{\Omega}{\left( e_h \right) ^2\mathrm{d}x}\le \int_{\Omega}{\frac{\partial \varepsilon _h}{\partial t}e_h\mathrm{d}x+}\sum_i{\sigma _i\int_{K_i}{\nu _i\frac{\partial v_h}{\partial x}A\left( u_h \right) \frac{\partial e_h}{\partial x}\mathrm{d}x}}.
$$
It can be calculated that
$$
\begin{aligned}
	\sum_i\sigma _i&\int_{K_i}{\nu _i\frac{\partial v_h}{\partial x}A\left( u_h \right) \frac{\partial e_h}{\partial x}\mathrm{d}x}\\
	&=\sum_i{\sigma _i\int_{K_i}{\nu _i\left[ \left( \frac{\partial v_h}{\partial x}A\left( u_h \right) -\frac{\partial u_h}{\partial x} \right) +\frac{\partial u_h}{\partial x} \right] \frac{\partial e_h}{\partial x}\mathrm{d}x}}\\
	&\le \sum_i{\sigma _i\left( \left\| \frac{\partial v_h}{\partial x}A\left( u_h \right) -\frac{\partial u_h}{\partial x} \right\| _{L^2\left( K_i \right)}+\left\| \frac{\partial u_h}{\partial x} \right\| _{L^2\left( K_i \right)} \right) \left\| \frac{\partial e_h}{\partial x} \right\| _{L^2\left( K_i \right)}}.
\end{aligned}
$$
According to the assumption and the definition that $v_h$ is the $(k+1)$-th interpolation of $v(u_h)$, we have
$$\begin{aligned}
\left\| \frac{\partial v_h}{\partial x}A\left( u_h \right) -\frac{\partial u_h}{\partial x} \right\| _{L^2\left( K_i \right)}&=\left\| \frac{\partial v_h}{\partial x}A\left( u_h \right) -\frac{\partial v\left( u_h \right)}{\partial x}A\left( u_h \right) \right\| _{L^2\left( K_i \right)}
\\
&\le \max_{x\in K_i} \left| A\left( u_h \right) \right| \, \left\| \frac{\partial v_h}{\partial x}-\frac{\partial v\left( u_h \right)}{\partial x} \right\| _{L^2\left( K_i \right)}
\\
&\lesssim \mathop {\max} \limits_{x\in K_i}\left| A\left( u_h \right) \right| \, \max_{x\in K_i} \left| \frac{\partial ^{k+1}v\left( u_h \right)}{\partial x^{k+1}} \right|h^{k+1/2}
\\
&\lesssim  h^{k+1/2}
\end{aligned}
$$
and
$$ \left\| \frac{\partial u_h}{\partial x} \right\|_{L^2(K_i)}\lesssim h^{1/2}. $$
Therefore, there still holds
$$
\sum_i{\sigma _i\int_{K_i}{\nu _i\frac{\partial v_h}{\partial x}A\left( u_h \right) \frac{\partial e_h}{\partial x}\mathrm{d}x}}\lesssim \left( \sum_i{\left( \sigma _i \right) ^2}h^{-1} \right) ^{1/2}\left\| e_h \right\| _{L^2\left( \Omega \right)}.
$$
The rest of the proof are similar.    
\end{Remark}

\subsection{1D system}

Next, we consider the one-dimensional system
\begin{equation}\label{sys1} 
\mathbf u_t +\mathbf f(\mathbf u)_x = \mathbf{0}. 
\end{equation}
where $\mathbf u:\mathbb R\to \mathbb R^p$ and $\mathbf f:\mathbb R^p \to\mathbb R^p$. 
For systems, we can also define the entropy pair $(U,F)$, which are still scalar functions with respect to $\mathbf u$ and satisfy $U'(\mathbf u)\mathbf f'(\mathbf u)=F'(\mathbf u)$. 
Moreover, denote
$$\mathbf v = \left(\dfrac{\partial U}{\partial \mathbf u}\right)^T\in\mathbb R^p,\quad 
\mathbf A(\mathbf u) = \left(\dfrac{\partial^2 U}{\partial \mathbf u^2}\right)^{-1}\in\mathbb R^{p\times p}, \quad
\mathbf B(\mathbf v) = \frac{\partial \mathbf{f}(\mathbf{u}(\mathbf{v}))}{\partial \mathbf{v}}\in\mathbb R^{p\times p}.$$
A strictly convex function $U$ serves as an entropy function if and only if $\mathbf A$ is a symmetric, positive-definite matrix and $\mathbf B$ is symmetric. Furthermore, as mentioned in \cite{chen2017entropy}, the change of variables $\mathbf{u} \mapsto \mathbf{v}$ should symmetrize the viscous term simultaneously such that the entropy conditions could meet the vanishing viscosity approach.

For example, we consider the compressible Euler equation
\begin{equation}\label{eq:1DEuler}
\left[ \begin{array}{c}
	\rho\\
	\rho u\\
	\mathcal{E}\\
\end{array} \right]_t +\left[ \begin{array}{c}
	\rho u\\
	\rho u^2+p\\
	u\left( \mathcal{E} +p \right)\\
\end{array} \right]_x = \mathbf{0}. 
\end{equation}
Here $\rho$, $u$ and $p$ are the density, velocity and pressure of the gas. $\mathcal{E}$ is the total energy, satisfying the equation of state for polytropic ideal gas 
$$\mathcal E = \frac{p}{\gamma - 1}+\frac{1}{2}\rho u^2,\quad \gamma = 1.4.$$
The physical specific entropy is $s = \ln(p\rho^{-\gamma})$. 
Harten \cite{HARTEN1983151} proved that there exists a family of entropy pairs that are related to $s$ and satisfy the symmetric conditions. However, if we also want to symmetrize the viscous term in the compressible Navier–Stokes equations with heat conduction \cite{HUGHES1986223}, there is only one choice of the entropy pair
$$ U = -\frac{\rho s}{\gamma - 1},\quad 
F = -\frac{\rho us}{\gamma - 1}.$$
Correspondingly,
$$ 
\mathbf{v}=\left[ \begin{array}{c}
	\dfrac{\gamma -s}{\gamma -1}-\dfrac{\rho u^2}{2p}\\
	\rho u/p\\
	-\rho /p\\
\end{array} \right],$$
$$\mathbf{A} =\left[ \begin{matrix}
	\rho&		m&		\displaystyle\frac{p}{\gamma -1}+\frac{m^2}{2\rho}\\
 \\
	m&		\displaystyle p+\frac{m^2}{\rho}&		\displaystyle\frac{m}{2\rho ^2}\left( m^2+\frac{2\gamma \rho p}{\gamma -1} \right)\\
 \\
	\displaystyle\frac{p}{\gamma -1}+\frac{m^2}{2\rho}\ \ &		\displaystyle\frac{m}{2\rho ^2}\left( m^2+\frac{2\gamma \rho p}{\gamma -1} \right)\ \ &		\displaystyle\frac{1}{4\rho ^3}\left( m^4+\frac{4\gamma m^2\rho p}{\gamma - 1}+\frac{4\gamma \rho^2 p^2}{(\gamma -1)^2} \right)\\
\end{matrix} \right],
$$
with $m=\rho u$.

To construct the NOES-DG scheme, the corresponding viscosity equation in cell $K_i$ is given as
\begin{equation}
\mathbf{u}_{t}+\mathbf{f}\left( \mathbf{u} \right) _x =\sigma_i\left( \nu _i\left( x \right) \mathbf A(\mathbf u)\mathbf{v}_x \right) _x,
\end{equation}
Here, $\sigma_i\geq0$ is scalar and $\nu_i(x)$ is defined the same as \eqref{eq:nu}.

Then, the NOES-DG scheme of the system \eqref{sys1} is: find $\mathbf u_h\in [V_h^k]^p$, such that for all $\mathbf w\in [V_h^k]^p$ and $K_i\in\mathcal K$,
\begin{equation}\label{scheme2}\begin{aligned}
\int_{K_i}{\frac{\partial \mathbf{u}_h}{\partial t}\cdot \mathbf{w}\mathrm{d}x}&=\int_{K_i}{\mathbf{f}\left( \mathbf{u}_h \right) \cdot \frac{\partial \mathbf{w}}{\partial x}\mathrm{d}x}-\mathbf{\hat{f}}_{i+1/2}\cdot\mathbf{w}_{i+1/2}^{-}+\mathbf{\hat{f}}_{i-1/2}\cdot\mathbf{w}_{i-1/2}^{+}
\\
&-\sigma_i\int_{K_i}{\nu _i\left( x \right) \left(  \frac{\partial \mathbf{v}_h}{\partial x} \right) ^T\mathbf{A}\left( \mathbf{u}_h \right) \frac{\partial \mathbf{w}}{\partial x}\mathrm{d}x}.
\end{aligned}
\end{equation}
Note that $\mathbf v_h$ must interpolate the value of $\mathbf v(\mathbf u_h)$ at the interfaces, i.e.
$$ \mathbf v_h|_{i+ 1/2}^{-} = \mathbf v(\mathbf u_h|_{i + 1/2}^-),\quad 
\mathbf v_h|_{i- 1/2}^{+} = \mathbf v(\mathbf u_h|_{i - 1/2}^+). $$
We do interpolation on each component based on the Gauss-Lobatto quadrature points again. 
The parameter $\sigma_{i}$ is determined by $\sigma _{i}=\max \left\{\sigma _{i}^{jump},  \sigma _{i}^{entropy}\right\}$ as well, where
\begin{equation*}
\begin{aligned}
\sigma_i^{jump}&=\max\limits_{1\le m\le p} \sigma_{i,m}^{jump},\\
\sigma_{i,m}^{jump}&=c_f\left( h\left\| [\![ u_{h,m} ]\!] \right\| _{\partial K_i}+\sum_{l=1}^k{l\left( l+1 \right) h^{l+1}\left\| [\![ \partial _{x}^{l}u_{h,m} ]\!] \right\| _{\partial K_i}} \right),
\\
\sigma _{i}^{entropy}&=\min \left\{ \max \left\{ \frac{F_i}{E_i},0 \right\} ,C\,\max_m \left\{ \sigma _{i,m}^{jump} \right\} \right\}, \\
E_i&=\int_{K_i}{\nu _i\left( x \right) \left( \frac{\partial \mathbf{v}_h}{\partial x} \right) ^T\mathbf{A}\left( \mathbf{u}_h \right) \frac{\partial \mathbf{v}_h}{\partial x}\mathrm{d}x},\\
F_i&=\hat{F}_{i+1/2}-\hat{F}_{i-1/2}-\mathbf{\hat{f}}_{i+1/2}^{C}\cdot \mathbf{v}_h|_{i+1/2}^{-}+\mathbf{\hat{f}}_{i-1/2}^{C}\cdot\mathbf{v}_h|_{i-1/2}^{+}+\int_{K_i}{\mathbf{f}\left( \mathbf{u}_h \right)\cdot\frac{\partial \mathbf{v}_h}{\partial x}\mathrm{d}x}.
\end{aligned}
\end{equation*}
Here, similar to the scalar case, $\sigma_i^{jump}$ is defined to measure discontinuity on each component without characteristic decomposition. For the Euler equations, we define the coefficient $c_f$ as $c_f = c_0\max\limits_{x\in K_i}\{ 1/H \}$ as in \cite{wei2024jump}, where $H = (\mathcal E+p)/\rho$ denotes the enthalpy, and $c_0$ is a constant. Note that $c_f$ can also be chosen as the scale-invariant version under the OEDG framework \cite{peng2024oedg}. Again, as $\sigma_{i}^{entropy}$ is not scale-invariant,  our scheme does not possess this property. 
The numerical entropy flux $\hat{F}$ is given as the central flux \eqref{fhat} as well. 
The numerical flux $\hat{\mathbf f}=\hat{\mathbf f}(\mathbf{u}_{L}, \mathbf{u}_{R})$ is split into the central part $\hat{\mathbf f}^C$ and the diffusive part $\hat{\mathbf f}^D$.
For example, for the Lax-Friedrichs flux,
$$
\mathbf{\hat{f}}^C=\frac{1}{2}\left( \mathbf{f}(\mathbf{u}_L) +\mathbf{f}(\mathbf{u}_R) \right) ,\quad 
\mathbf{\hat{f}}^D=-\frac{\alpha}{2}\left( \mathbf{u}_R-\mathbf{u}_L \right),\quad 
\alpha=\max\limits_{\mathbf u, i}\left|\lambda_i \left( \frac{\partial \mathbf f}{\partial \mathbf u} \right) \right|,
$$
where $\lambda_i$ is the $i$-th eigenvalue of the given matrix. 
For Euler equations, we can also take the numerical flux $\mathbf{\hat{f}}$ as HLL/HLLC flux \cite{toro2019hllc}. Correspondingly, for HLL flux, we can set
\begin{equation*}
   \mathbf{\hat{f}}^C=\begin{cases}
	\mathbf{f}(\mathbf{u}_L),&S_L\ge 0,\\
	\displaystyle\frac{S_L\mathbf{f}(\mathbf{u}_R)-S_R\mathbf{f}(\mathbf{u}_L)}{S_R-S_L},&S_L<0<S_R,\\
	\mathbf{f}(\mathbf{u}_R),& S_R\le 0,\\
\end{cases} \quad 
\mathbf{\hat{f}}^D=\begin{cases}
	0,&S_L\ge 0,\\
	\displaystyle\frac{S_RS_L\left( \mathbf{u}_R-\mathbf{u}_L \right)}{S_R-S_L}, & S_L<0<S_R,\\
	0, & S_R\le 0,\\
\end{cases}
\end{equation*}
where $S_R$ and $S_L$ are the local estimates of maximal and minimal wave speed at the interface, respectively. And for HLLC flux, we can set
\begin{equation*}
    \hat {\mathbf f}^C = \hat {\mathbf f} - \hat {\mathbf f}^D,\quad \mathbf{\hat{f}}^{D}=\begin{cases}
	0, & S_L\ge 0,\\
	\displaystyle\frac{S_*S_L\left( \mathbf{u}_R-\mathbf{u}_L \right)}{S_*-S_L}, &S_L<0\le S_*,\\
	\displaystyle \frac{S_RS_*\left( \mathbf{u}_R-\mathbf{u}_L \right)}{S_R-S_*}, &S_*<0<S_R,\\
	0, &S_R\le 0,\\
\end{cases} 
\end{equation*} 
where $S_*$ is an estimate of contact wave speed calculated by $\mathbf u_R,\mathbf u_L,S_R,S_L$. The specific formula of $\hat {\mathbf f}$ and $S_*$ can be found in \cite{toro2019hllc}.

Following the same approach as in proving Theorem \ref{esthm}, we can similarly establish that the scheme \eqref{scheme2} is entropy stable for the hyperbolic system \eqref{sys1} with periodic or compact boundaries.

\begin{Thm}
If 
\begin{equation}\label{sigma1sys}
 \sigma_i^{entropy} < C \max\limits_{m}\{\sigma_{i,m}^{jump}\},\quad \forall K_i\in\mathcal K,
\end{equation}
then the NOES-DG scheme \eqref{scheme2} is entropy stable in the sense of
\begin{equation}\label{es0sys}
\int_{\Omega}{\frac{\partial \mathbf u_h}{\partial t}\cdot \mathbf v_h\mathrm{d}x} \le 0.
\end{equation}
\end{Thm}
\begin{proof}
For Lax-Friedrichs flux, HLL flux, or HLLC flux, we can write
$$ \hat{\mathbf f}^D_{i+1/2}=-S_{i+1/2}(\mathbf u_h|_{i+1/2}^+-\mathbf u_h|_{i+1/2}^-),\quad S_{i+1/2}\ge 0. $$
Hence, taking $\mathbf w=\mathbf v_h$ in \eqref{scheme2} and summing of $i$, we have
\begin{align*}
\int_{\Omega}{\frac{\partial \mathbf{u}_h}{\partial t}\cdot \mathbf{v}_h\mathrm{d}x}\le& -\sum_i{\left( \hat{\mathbf{f}}_{i+1/2}^{D}\cdot \mathbf{v}_h|_{i+1/2}^{-}-\hat{\mathbf{f}}_{i-1/2}^{D}\cdot \mathbf{v}_h|_{i-1/2}^{+} \right)}
\\
=&-\sum_i{S_{i+1/2}\left( \mathbf{u}_h|_{i+1/2}^{+}-\mathbf{u}_h|_{i+1/2}^{-} \right) \cdot \left( \mathbf{v}_h|_{i+1/2}^{+}-\mathbf{v}_h|_{i+1/2}^{-} \right)}
\\
=&-\sum_i{S_{i+1/2}\left( \mathbf{u}_h|_{i+1/2}^{+}-\mathbf{u}_h|_{i+1/2}^{-} \right) \cdot \left( \mathbf{v}\left( \mathbf{u}_h|_{i+1/2}^{+} \right) -\mathbf{v}\left( \mathbf{u}_h|_{i+1/2}^{-} \right) \right)}
\\
=&-\sum_i{S_{i+1/2}\int_0^1{\left( \mathbf{u}_h|_{i+1/2}^{+}-\mathbf{u}_h|_{i+1/2}^{-} \right) ^T}}\mathbf{A}\left( \mathbf{u}_{\theta} \right) \left( \mathbf{u}_h|_{i+1/2}^{+}-\mathbf{u}_h|_{i+1/2}^{-} \right) \mathrm{d}\theta 
\\ \le&\  0.
\end{align*}
Here, $\mathbf u_\theta=\mathbf u_h|_{i+1/2}^-+\theta(\mathbf u_h|_{i+1/2}^+-\mathbf u_h|_{i+1/2}^-)$.

\end{proof}

\subsection{Two-dimensional system}

In this section, we consider a two-dimensional system
\begin{equation}\label{hcl2}
\begin{cases}
    \mathbf u_t+\nabla \cdot \mathbf f\left( \mathbf u \right) = \mathbf{0},
    \quad \mathbf x\in \Omega\subset\mathbb R^2,\\
    \mathbf u\left( \mathbf x,0 \right) =\mathbf u_0\left( \mathbf x \right),\\
\end{cases}
\end{equation}
with
$$\mathbf u:\mathbb R^2\to\mathbb R^p,\quad \mathbf f(\mathbf u)=(\mathbf f_1(\mathbf u),\mathbf f_2(\mathbf u)):\mathbb R^p\to \mathbb R^{p\times 2}.$$
Here, we consider the domain is a rectangular $\Omega=[a_x,b_x]\times [a_y, b_y]$ with either periodic or compact support in each direction. 
The entropy pair $(U,\bm{\mathcal F})$, where $\bm{\mathcal F} = (F_1,F_2)$, should also satisfying the symmetric condition. 

Assume the computational domain $\Omega$ is divided into a Cartesian mesh $\mathcal K=\{K_{ij}: K_{ij}=[x_{i-1/2}, x_{i+1/2}]\times [y_{j-1/2}, y_{j+1/2}]\}$. 
And the finite element space is defined as piecewise $Q^k$ polynomial
$$
V_{h}^{k}=\left\{ w\left( x,y \right) :w\left( x,y \right) |_{K_{ij}} \in Q^k(K),\forall K_{ij}\in \mathcal{K} \right\}.
$$
Then, the one-dimensional framework could be directly applied to rectangular meshes through tensor product, with appropriate adjustment on the additional local viscosity in cell $K$.

In particular, the local viscosity form of \eqref{hcl2} on a cell $K\in\mathcal K$ is defined as
\begin{equation}
\mathbf u_t+\nabla \cdot \mathbf f\left( \mathbf u \right) =\sigma_K\nabla\cdot\left( \nu _K\mathbf A(\mathbf u) \nabla\mathbf v(\mathbf u) \right).
\end{equation}
We take 
$$
\nu _K\left( x,y \right) =\left(1 - X^2\left( x,y \right) \right) \left( 1-Y^2\left( x,y \right) \right), 
$$
to satisfy the condition
$$
\nu _K\left( \mathbf x \right) =0,\,\, \mathbf x \in \partial K,\quad \text{and} \quad 
\nu _K\left( \mathbf x \right) >0,\,\, \mathbf x \in K\backslash \partial K.
$$
Here, $(x,y)\to (X,Y)$ is the map from physical cell $K$ to the reference cell $K_0=\{(X,Y):-1<X,Y<1\}.$

The NOES-DG scheme of \eqref{hcl2} reads: Find $\mathbf u_h\in [V_h^k]^p$, such that for all $\mathbf w\in [V_h^k]^p$ and  $K\in\mathcal K$,
\begin{equation}\label{scheme3}
\begin{aligned}
\int_K{\mathbf u_h \cdot \mathbf w\, \mathrm{d}K}
-\int_K{\mathbf f\left( \mathbf u_h \right) : \nabla \mathbf w\, \mathrm{d}K}
&+\int_{\partial K}{\hat{\mathbf f}\left( \mathbf u_h^{int},\mathbf u_h^{ext},\mathbf{n} \right)\cdot \mathbf w\, \mathrm{d}S}
\\&=-\sigma_K\int_K{\nu _K(\mathbf A\left( \mathbf u_h \right)\nabla \mathbf v_h): \nabla \mathbf w\, \mathrm{d}K}.\end{aligned}
\end{equation}
Here, $\hat {\mathbf f}$ is the numerical flux consistent with $\mathbf f(\mathbf u)\cdot\mathbf n$, and $\mathbf n$ is the unit outer normal vector of $K$. We can take $\mathbf v_h$ as the two-dimensional Gauss-Lobatto interpolation of $\mathbf v(\mathbf u_h)$ as well. Moreover, the parameter is given by $\sigma_{K}=\max\{\sigma_{K}^{jump},\sigma_K^{entropy}\}$,
\begin{equation}
\begin{aligned}
\sigma_K^{jump}&=\max\limits_{1\le m\le p}\sigma_{K,m}^{jump},
\\\sigma _{K,m}^{jump}&=c_f\left( h_K\left\| [\![ u_{h,m} ]\!] \right\| _{\partial K}+\sum_{\left|l\right|=1}^k{\left|l\right|\left( \left|l\right|+1 \right) h_K ^{\left|l\right|+1}\left\| [\![ \partial^{l}u_{h,m} ]\!] \right\| _{\partial K}} \right),\ l\in\mathbb N^d,\\
\sigma _{K}^{entropy}& =\min \left\{ \max \left\{ \frac{F_K}{E_K},0 \right\} ,C \max_m \left\{ \sigma _{K,m}^{jump} \right\} \right\}, \\
E_K&=\int_K{\nu _K\left( \mathbf{A}\left( \mathbf{u}_h \right) \nabla \mathbf{v}_h \right) :\nabla \mathbf{v}_h\mathrm{d}K}
\\
F_K&=\int_{\partial K}{\hat{\boldsymbol{\mathcal F}}\cdot \mathbf{n}\,\mathrm{d}S}-\int_{\partial K}{\left( \mathbf{\hat{f}}^C\cdot \mathbf{n} \right) \cdot \mathbf{v}_{h}^{int}\mathrm{d}S}+\int_K{\mathbf{f}\left( \mathbf{u}_h \right) :\nabla \mathbf{v}_h\mathrm{d}K},
\end{aligned}
\end{equation}
where $\hat{\boldsymbol{\mathcal F}}=(\hat F_1,\hat F_2)^T$.

In this paper, we consider the two-dimensional Euler equation, where
$$
\mathbf{u}=\left[ \begin{array}{c}
	\rho\\
	\rho u\\
    \rho v\\
	\mathcal{E}\\
\end{array} \right] ,\quad \mathbf{f}_1\left( \mathbf{u} \right) =\left[ \begin{array}{c}
	\rho u\\
	\rho u^2+p\\
    \rho uv\\
	u\left( \mathcal{E} +p \right)\\
\end{array} \right],\quad \mathbf f_2(\mathbf u) = \left[ \begin{array}{c}
	\rho v\\
	\rho uv\\
    \rho v^2 + p\\
	v\left( \mathcal{E} +p \right)\end{array} \right]
$$
with $$\mathcal E = \frac{p}{\gamma - 1}+\frac{1}{2}\rho (u^2 + v^2),\quad \gamma = 1.4.$$
And the entropy pair is
$$ U = -\frac{\rho s}{\gamma - 1},\quad F_1 = -\frac{\rho us}{\gamma - 1},\quad F_2 = -\frac{\rho vs}{\gamma - 1}.$$
Further,
$$\mathbf{v}=\left[ \begin{array}{c}
	\dfrac{\gamma -s}{\gamma -1}-\dfrac{\rho (u^2 + v^2)}{2p}\\
	\rho u/p\\
    \rho v/p\\
	-\rho /p\\
\end{array} \right],
$$
and
$$
\mathbf{A} =\left[ \begin{matrix}
	\rho&		m_1& m_2&		\displaystyle\frac{p}{\gamma -1}+\frac{M}{2\rho}\\
 \\
	m_1&		\displaystyle p+\frac{m_1^2}{\rho}&	\displaystyle\frac{m_1m_2}{\rho}&	\displaystyle\frac{m_1}{2\rho ^2}\left( M+\frac{2\gamma \rho p}{\gamma -1} \right)\\
 \\
	m_2&	\displaystyle\frac{m_1m_2}{\rho}&	\displaystyle p+\frac{m_2^2}{\rho}&		\displaystyle\frac{m_2}{2\rho ^2}\left( M+\frac{2\gamma \rho p}{\gamma -1} \right)\\
 \\
	\displaystyle\frac{p}{\gamma -1}+\frac{M}{2\rho}\ \ &		\displaystyle\frac{m_1}{2\rho ^2}\left( M+\frac{2\gamma \rho p}{\gamma -1} \right)\ \ &		\displaystyle\frac{m_2}{2\rho ^2}\left( M+\frac{2\gamma \rho p}{\gamma -1} \right)\ \ & A_{44}\\
\end{matrix} \right],
$$
with
$$ m_1 = \rho u,\quad m_2 = \rho v,\quad M=m_1^2 + m_2^2,\quad A_{44} = \frac{1}{4\rho ^3}\left( M^2+\frac{4\gamma M\rho p}{\gamma -1}+\frac{4\gamma \rho^2p^2}{(\gamma - 1)^2} \right).  $$
The flux splitting formula is given along each direction, which is the same as that in 1D.

\begin{Remark}
For triangular meshes, it is usually to use $P^k$ basis in numerical implementation. In this case we can also define $\mathbf v_h$ interpolating $\mathbf v(\mathbf u_h)$ on Gauss-Lobatto points along the edges. In general, for a $P^k$ polynomial we have $(k+1)(k+2)/2$ degrees of freedom, and the number of conditions of interpolation is $3k$.  
Therefore, the interpolation exists for all $k>0$ and is unique for $k = 1,2$. When $k>2$, there are some extra degrees of freedom to be determined, we may let $v_h$ interpolate some other points such as \cite{hesthaven1998electrostatics}. The method can also extend to other types of meshes, and it will be left to future work. 
\end{Remark}

\section{Fully discretization}

\subsection{Integral factor SSP-RK method}

The PDE can be solved in the MOL framework, that is after DG discretization the semi-discrete scheme can be treated as an ODE system, and then, we can apply a known ODE solver on it to obtain the fully discrete scheme.
However, for numerical solutions containing discontinuities, the additional viscous term becomes stiff, resulting in a severely restricted time step when using explicit time discretization.
To overcome this difficulty, we will employ the integrating factor strong-stability-preserving Runge-Kutta (IF-SSP-RK) method introduced in \cite{chen2021krylov}. 
Here, we only discuss the one-dimensional scalar case for brevity. 

Suppose we take the Legendre polynomials $\{\phi_{i,l}\}$ as the basis functions of $V^k_h$ in $K_i$, where $\phi_{i,l}$ is polynomials of degree $l$. 
Then, the numerical solution can be presented as
$$
u_h(x,t)|_{K_i}=\sum_{l=0}^k{u_{i}^{l}(t) \, \phi_{i,l}(x)}.
$$
We align the coefficients as a vector
$$
\mathbf{u}_h=\left( u_{1}^{0},\cdots ,u_{1}^{k},u_{2}^{0},\cdots ,u_{2}^{k},\cdots ,
u_{N}^{0},\cdots ,u_{N}^{k} \right) ^T\in\mathbb R^{(k+1)N}.
$$
After the DG spatial discretization, we obtain the following ODE system
\begin{equation}\label{IIF1}
\frac{\mathrm{d}}{\mathrm{d}t}\mathbf{u}_h=\mathcal{L} _h\left( \mathbf{u}_h \right) +\mathcal{L} _{h}^{v}\left( \mathbf{u}_h \right), 
\end{equation}
where $\mathcal L_h^v(\mathbf u_h)$ corresponding to the artificial viscosity, which may be a stiff term. 

To introduce the IF-SSP-RK method, we begin with a specific case that this stiff term is semilinear, i.e. $\mathcal L^v_h(\mathbf u_h) = \mathbf C \mathbf u_h$, where $\mathbf C$ is a constant matrix.
Then we can multiply the factor $e^{-\mathbf Ct}$ on both sides and get that
$$
\frac{\mathrm{d}}{\mathrm{d}t}\left(e^{-\mathbf Ct} \mathbf{u}_h\right) =e^{-\mathbf Ct}\mathcal{L} _h\left( \mathbf{u}_h \right). 
$$
Then the IF-SSP-RK schemes can be obtained via applying RK scheme on the above ODE. 
In particular, the third-order IF-SSP-RK method \cite{chen2021krylov} is given as following
\begin{equation}\label{IFSSP}  
\begin{aligned}
\mathbf{u}^{\,\left( 1 \right)}_h&=\frac{1}{2}e^{\frac{2}{3}\mathbf C \Delta t}\left( 2\mathbf{u}^{\,n}_h+\frac{4}{3}\Delta t \mathcal{L}_h\left( \mathbf{u}_{h}^{n} \right) \right), 
\\
\mathbf{u}_h^{\,\left( 2 \right)}&=\frac{2}{3}e^{\frac{2}{3} \mathbf C \Delta t}\mathbf{u}_{h}^{\,n}+\frac{1}{3}\left( \mathbf{u}_h^{\,\left( 1 \right)}+\frac{4}{3}\Delta t \mathcal{L}_h\left( \mathbf{u}_h^{\,\left( 1 \right)} \right) \right), 
\\
\mathbf{u}_{h}^{\,n+1}&=\frac{59}{129}e^{\mathbf C\Delta t} \mathbf{u}_{h}^{\,n} 
+\frac{15}{129}e^{\mathbf C\Delta t}\left( \mathbf{u}_h^{\,n} +\frac{4}{3}\Delta t \mathcal{L}_h\left( \mathbf{u}_{h}^{\,n} \right) \right) 
+\frac{27}{64}\left( \mathbf{u}_h^{\,\left( 2 \right)}+\frac{4}{3}\Delta t \mathcal{L}_h \left( \mathbf{u}_h^{\,\left( 2 \right)} \right) \right). 
\end{aligned}
\end{equation}
\noindent
If the stiff term is nonlinear, we can use $\mathbf C_n \mathbf u_h$ to approximate $\mathcal L_h^v(\mathbf u_h)$, where $\mathbf C_n$ is a constant matrix. Then, rewrite the ODE system \eqref{IIF1} as
\begin{equation}\label{IIF2}
\frac{\mathrm{d}}{\mathrm{d}t}\mathbf{u}_h=\mathcal{L} _h\left( \mathbf{u}_h \right) +\mathcal{L} _{h}^{v}\left( \mathbf{u}_h \right) - \mathbf C_n\mathbf u_h + \mathbf C_n \mathbf u_h.
\end{equation}
Denote $\tilde{\mathcal L}_h(\mathbf u_h):=\mathcal L_h(\mathbf u_h)+\mathcal L_h^v(\mathbf u_h)- \mathbf C_n\mathbf u_h$. Then we can get the fully-discrete scheme via the same processing replacing ${\mathcal L}_h(\mathbf u_h)$ and $\mathbf C$ in \eqref{IFSSP} by $\tilde{\mathcal L}_h(\mathbf u_h)$ and $\mathbf C_n$ respectively.

Notice that in our scheme, the viscosity term $\displaystyle\int_{K_i} \nu_i(x)(v_h)_x A(u_h) w_x\mathrm dx$ is a nonlinear local form. 
To choose a simple form of $\mathbf C_n$, we first consider a special case with square entropy $U(u)=u^2/2$, and denote the corresponding matrix $\mathbf C_n$ as $\mathbf C_n^{square}$. 
In this case $A(u_h)(v_h)_x= (u_h)_x$ and the matrix $\mathbf C_n^{square}$ can be given exactly satisfying $\mathbf C_n^{square}\mathbf u_h=\mathcal L^v_h(\mathbf u_h)$. To be specific, according to the local property, $\mathbf C_n^{square}$ is a blocked diagonal matrix.
Additionally, with the function $\nu_i(x) = 1 - X_i^2(x)$, the basis function $\phi_{i,l}$ satisfies a Sturm-Liouville equation 
$$\frac{h^2}{4}\frac{\mathrm d}{\mathrm dx}\left(\nu_i\frac{\mathrm d}{\mathrm dx}\phi_{i,l}\right) + l(l+1)\phi_{i,l}=0,$$ 
leading to
$$
\int_{K_i}{\nu _i\left( \phi _{i,l} \right)_{x}(\phi_{i,q})_x\mathrm{d}x}=\delta_{lq}\frac{2}{h}l\left( l+1 \right) m_l,\quad m_l=
\frac{2}{2l+1}\left( \frac{2^l\left( l! \right) ^2}{\left( 2l \right) !} \right) ^2.
$$
Therefore, we can derive that $\mathbf C_n^{square}$ is a diagonal matrix
\begin{equation}\label{eq:C}
\mathbf C_n^{square}=\frac{1}{h^2}\mathrm{diag}\left\{ \sigma _1,\sigma _2,\cdots ,\sigma _N \right\} \otimes \mathrm{diag}\left\{ 0,4,\cdots ,2k\left( k+1 \right) \right\},
\end{equation}
where $\otimes$ is the Kronecker product $\mathbf A\otimes \mathbf B = (a_{ij}\mathbf B)$. 
For a general entropy function $U$, since $ A(u_h)(v_h)_x\approx (u_h)_x $, we can also set $\mathbf C_n=\mathbf C_n^{square}$ \eqref{eq:C} in \eqref{IIF2}. 
As a consequence, the matrix $e^{\mathbf C_n\Delta t}$ can be obtained directly, making the algorithm easy to implement.

\subsection{Positivity-preserving limiter}

For certain extreme cases, negative density or pressure may appear in the numerical solution produced by the proposed method.  In such cases, the system \eqref{sys1} becomes non-hyperbolic, rendering the initial value problem ill-posed. 
To overcome the difficulty, we propose to utilize the PP limiter \cite{zhang2010positivity}. Denote $\bar{\mathbf{u}}_{i}^{n}$ as the cell average of the numerical solution $u_h$ in the cell $K_i$ at time $t_n$:
$$\bar{\mathbf{u}}_{i}^{n} = \frac{1}{h} \int_{K_i} \mathbf u_h(x,t^n) \mathrm dx.$$
Note that the cell averages satisfy the same equation as in the standard RKDG formulation since the elements in the corresponding positions of $\mathbf C$ are all zero.  Therefore, although the time discretization used in this work differs from the standard SSP-RK method employed in the original work \cite{zhang2010positivity}, both approaches are essentially equivalent when applying the PP limiter. 

First, we briefly review  the PP limiter developed in \cite{zhang2010positivity} for the one-dimensional Euler equations \eqref{sys1} based on the standard SSP-RKDG method. 
Let $\mathcal G\subset \mathbb R^3$ denote the admissible set
$$ \mathcal G:= \left\{(\rho,m,\mathcal E): \rho>0\ \ \mathrm{and}\ \ p=(\gamma - 1)\left(\mathcal E-\frac{1}{2}\frac{m^2}{\rho}\right)>0\right\},$$
which is convex. 
It was proved that the first order DG scheme ($k=0$) coupled with the forward Euler method is PP under a certain CFL condition $\frac{\Delta t}{h} \left\| (\left|u\right|+c)\right\|_{L^\infty(\Omega)}\le \alpha_0$. That is if $\bar{\mathbf{u}}_{i}^{n}\in \mathcal{G}$, then $\bar{\mathbf{u}}_{i}^{n+1}\in \mathcal{G}$. 
Moreover, \cite{zhang2010positivity} shows that a high order DG scheme,  when coupled with the forward Euler method,  can ensure $\bar{\mathbf{u}}_{i}^{n+1}\in \mathcal{G}$ under a suitable CFL condition, provided that the nodal values of the approximate polynomials at the Gauss-Lobatto quadrature points remain within $\mathcal{G}$. 
Building on this, authors designed a simple scaling-based PP limiter that preserves the original order of accuracy. 
Since an SSP-RK method is a convex combination of forward Euler steps, the PP property still holds for an SSP-RKDG method as long as the limiter is applied at each stage of the SSP-RK method.

In the following, we will apply the PP-limiter on the NOES-DG scheme \eqref{scheme2}. Let $\{(\hat{x}_i^q, \hat\omega_i^q)\}_{q=1}^{k+1}$ be the Gauss-Lobatto points and quadrature weights on a interval $K_i$. 
Note that the cell average of the numerical solution by the  NOES-DG scheme coupled with the IF forward Euler method satisfies the following formulation   
\begin{equation*}\label{eq:NOES_Euler}
\begin{aligned}
\bar{\mathbf{u}}_i^{n+1}=\bar{\mathbf{u}}_i^{n}- \frac{\Delta t}{h} \cdot \left(
\mathbf{\hat{f}}_{i+1/2}^{n}-\mathbf{\hat{f}}_{i-1/2}^{n}\right)
:= \bar{\mathbf{u}}_i^{n} + \Delta t \mathcal{L}_{i}^{ave}(\mathbf{u}_h^n),
\end{aligned}
\end{equation*}
which is the same as that of the traditional DG scheme coupled with the forward Euler method. 
Hence, based on such equivalence, we can directly employ the PP-limiter in \cite{zhang2010positivity} for our NOES-DG scheme.  
 In particular, the DG solution is rescaled  using  two scaling parameters $\theta_i^{(1)}, \theta_i^{(2)} \in[0,1]$ on each cell $K_i$:
\begin{equation}\label{pp}
\begin{aligned}
\tilde{\rho}^{n}_h|_{K_i} &= \left(1-\theta_i^{(1)}\right){\bar\rho}_i^n + \theta_i^{(1)} \rho^n_h|_{K_i},
\\ 
\left. \left( \begin{array}{c}
	\rho_h\\
	m_h\\
	\mathcal{E}_h\\
\end{array} \right) ^{n,\left( mod \right)} \right|_{K_i}&=\left( 1-\theta _{i}^{(2)} \right) \left( \begin{array}{c}
	\bar{\rho}_{i}^{n}\\
	\bar{m}_{i}^{n}\\
	\bar{\mathcal{E}}_{i}^{n}\\
\end{array} \right) +\theta _{i}^{(2)}\left. \left( \begin{array}{c}
	\tilde{\rho}_h\\
	m_h\\
	\mathcal{E}_h\\
\end{array} \right) ^n \right|_{K_i} 
\end{aligned}
\end{equation}
to guarantee $\tilde{\rho}_h^n(\hat{x}^q_i)\geq0$ and $\mathbf{u}_h^{n,(mod)}(\hat{x}^q_i) \in \mathcal{G}$. 
The values of $\theta_i^{(1)}$ and $\theta_i^{(2)}$ are given by
\begin{equation} 
\theta_i^{(1)} = \min\left\{ \frac{\bar\rho^n_i-\varepsilon}{\bar\rho^n_i-\rho_{m}},1 \right\},\qquad 
\rho_m=\min\limits_{1\le q\le k+1}\rho^n_h(\hat{x}_i^q), 
\end{equation}
\begin{equation}
\theta_i^{(2)}=\min\limits_{1\le q\le k+1}\{t_q\},\qquad t_q=\begin{cases}
	1,\quad p\left( \tilde{\mathbf{u}}_h\left( \hat{x}_{i}^{q} \right) \right) \ge \varepsilon,\\
	\tilde{t}_q, \quad p\left( \tilde{\mathbf{u}}_h\left( \hat{x}_{i}^{q} \right) \right) <\varepsilon,\\
\end{cases}
\end{equation}
where $\tilde{\mathbf u}_h=(\tilde\rho_h^n,m_h^n,\mathcal E_h^n)$, and $\tilde t_q$ is solved from the following nonlinear equation
$$ p((1-\tilde t_q)\bar{\mathbf u}_i + \tilde t_q\tilde{\mathbf u}_h(\hat{x}_i^q)) = \varepsilon,$$
where the parameter $\varepsilon$ is a small number, and we set $\varepsilon=10^{-13}$ as suggested in \cite{zhang2010positivity}. 
Finally, we note that the scaling PP-limiter \eqref{pp} maintains the cell averages. Moreover, as shown in \cite{zhang2010positivity}, it preserves the original high order of accuracy for smooth solutions.

\begin{Remark}
For the two-dimensional case, the quadrature points for the PP limiter on a rectangular cell are defined as $\mathbb S_{2D}=(\mathbb S_i^x\times \hat{\mathbb S}_j^y) \bigcup (\hat{\mathbb S}_i^x\times \mathbb S_j^y)$, where $\mathbb S_i^x$ and $\hat{\mathbb S}_i^x$ denote the Gauss quadrature points and Gauss-Lobatto quadrature points on $[x_{i-1/2},x_{i+1/2}]$, respectively. The scaling parameters $\theta^{(1)}_{ij}, \theta^{(2)}_{ij}$ are calculated at these $2(k+1)^2$ points, see \cite{zhang2010positivity}. For triangular meshes, the main difference lies in the  quadrature rule, and the details can be found in \cite{zhang2012maximum}.
\end{Remark}

\section{Numerical tests}

In this section, we present the results of our numerical experiments for the schemes described in the previous sections.
In particular, for smooth problems, we do accuracy test with $k = 1,2,3$. And for those non-smooth problems, without loss of generality, we only show results with $k = 2$. The domain is divided into uniform meshes. 
The third-order IF-SSP-RK method is used for time discretization, with time step
$$ \Delta t =  \frac{\mathrm{CFL}}{\max\{\alpha_x,\alpha_y\}} h,$$
and $\mathrm{CFL}$ is taken as $0.2$. 
Unless otherwise noted, we take the parameter $C = 50$ for all tests. 
 Moreover, $c_0$ is a free parameter, and the optimal value is problem-dependent. Our selection principle is to choose a small value while preventing spurious oscillations. Based on our experience, a suitable range is  $0.1 < c_0 < 10$. Further investigation into these parameters, including the scale-invariant properties, will be considered in future work.
Time evolution of the total entropy will be given for some problems, which is defined as $\int_{\Omega} U(u_h) \mathrm dx$.

\subsection{One-dimensional tests}

\begin{Ex}
\textbf{(Linear equation.)}
\label{ex:linear}
\end{Ex}

We first consider the linear equation
$$ u_t+u_x=0, \quad x\in\Omega=[0,2\pi], $$
with periodic boundary conditions. The exact solution is $u(x,t) = u(x-t,0)$.
In this example, we choose the entropy function as $U(u) = e^u$. 

First, we test the accuracy of the scheme with the smooth initial data $u(x,0) = \sin(x)$. In Table \ref{tab1}, we present the errors and orders of accuracy at $T=2\pi$. We can observe the optimal convergence rates for $k=1, 2, 3$. 
For this test problem, we also provide the CPU cost comparison  between the proposed NOES-DG scheme and standard DG scheme for $k=2$ in Table \ref{tabCPU}. It is observed that the percentage of increase in cost is negligible (less than 4\%), highlighting the efficiency of the proposed technique.

Next, we test this problem with the non-smooth initial data
$$
u\left( x,0 \right) =\begin{cases}
	1, \quad 0.5\pi <x<1.5\pi,\\
	0, \quad \mathrm{otherwise}.\\
\end{cases}
$$
In Figure \ref{figstep}, we present the result at $T = 2\pi$ with $N= 200$ cells. It can be seen that there is no obvious oscillation near the discontinuity, demonstrating that our proposed scheme can control spurious oscillation efficiently. The development of total entropy is shown in Figure \ref{figU1}(a), indicating that this quantity does not increase with time during the simulation. 
 
\begin{table}[htb!]
        \centering
        \caption{Example \ref{ex:linear}: One dimensional linear equation with smooth initial data. Errors and orders at final time $T = 2\pi$.}
	\setlength{\tabcolsep}{3.2mm}{
		\begin{tabular}{|c|c|cc|cc|cc|}
			\hline 
   &$N$ & $L^1$ error & order & $L^2$ error & order & $L^\infty$ error & order  \\ \hline 
   \multirow{4}{*}{$k=1$} 
    &64  &3.39e-04 & -- &4.35e-04 & -- &1.09e-03 & -- \\
   &128  &8.09e-05 & 2.07 &1.05e-04 &2.06 &3.05e-04 &1.83 \\
   &256  &1.99e-05 & 2.02 &2.60e-05 &2.01 &7.96e-05 &1.94 \\
   &512  &4.94e-06 & 2.01 &6.48e-06 &2.00 &2.02e-05 &1.97 \\
			\hline
   \multirow{4}{*}{$k=2$} 
    & 64  &4.38e-06 & -- &5.15e-06 & -- &1.27e-05 & -- \\
   & 128  &3.90e-07 & 3.49 &4.75e-07 &3.44 &1.29e-06 &3.30 \\
   & 256  &4.65e-08 & 3.07 &5.66e-08 &3.07 &1.57e-07 &3.05 \\
   & 512  &5.75e-09 & 3.02 &6.99e-09 &3.02 &1.95e-08 &3.01 \\ 
		\hline 
  \multirow{4}{*}{$k=3$}
    &64  &6.13e-08 & -- &7.29e-08 & -- &1.60e-07 & -- \\
   &128  &1.98e-09 & 4.95 &2.60e-09 &4.81 &6.30e-09 &4.67 \\
   &256  &1.08e-10 & 4.19 &1.45e-10 &4.16 &3.82e-10 &4.04 \\ 
   &512  &7.10e-12 & 3.93 &9.16e-12 &3.98 &2.27e-11 &4.07 \\
   \hline
	\end{tabular}} 
    \label{tab1}
\end{table}

\begin{table}[htb!]
        \centering
        \caption{Example \ref{ex:linear}: One dimensional linear equation with a smooth initial data. The comparison of CPU cost (in seconds) for $k=2$ up to $T=2\pi$.}
	\setlength{\tabcolsep}{3.2mm}{
		\begin{tabular}{|c|c|c|c|}
			\hline    $N$ & Standard & NOES & Ratio \\  
            \hline 64 & 0.3996s & 0.4116s & 1:1.0300\\
            \hline 128 & 1.3884s & 1.4434s & 1:1.0396\\
            \hline 256 & 5.3796s & 5.4651s & 1:1.0159\\
            \hline 512 & 21.1410s & 21.2178s & 1:1.0036\\
            \hline
	\end{tabular}}
    \label{tabCPU}
\end{table}

\begin{figure}[htbp!]
	\centering
    \subfigure[Overview.]{
		\includegraphics[width=0.45\linewidth]{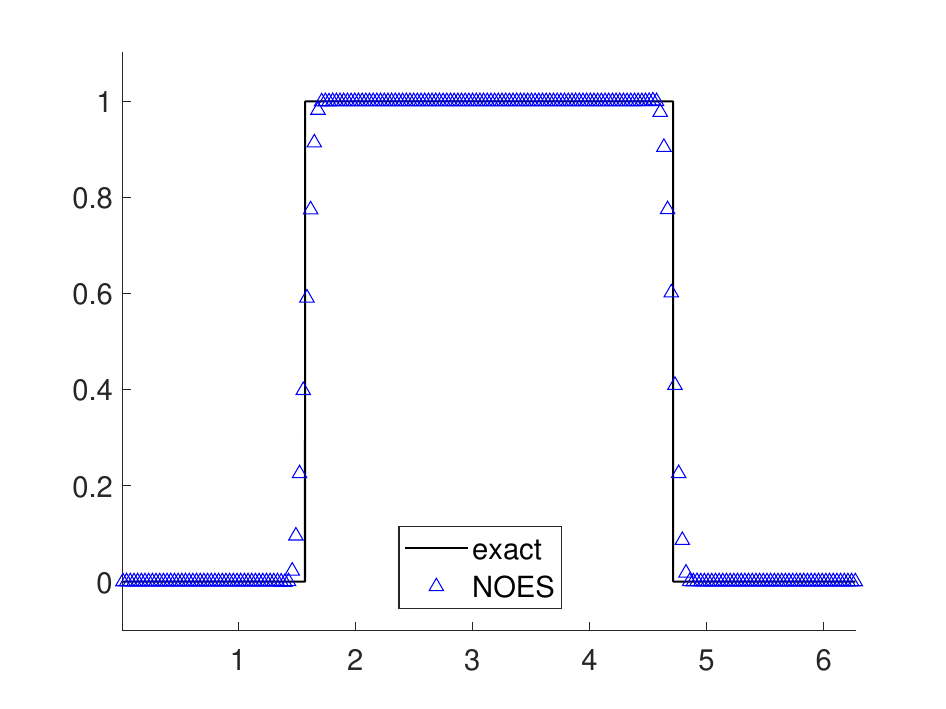}}
	\subfigure[Zoomed in results.]{
		\includegraphics[width=0.45\linewidth]{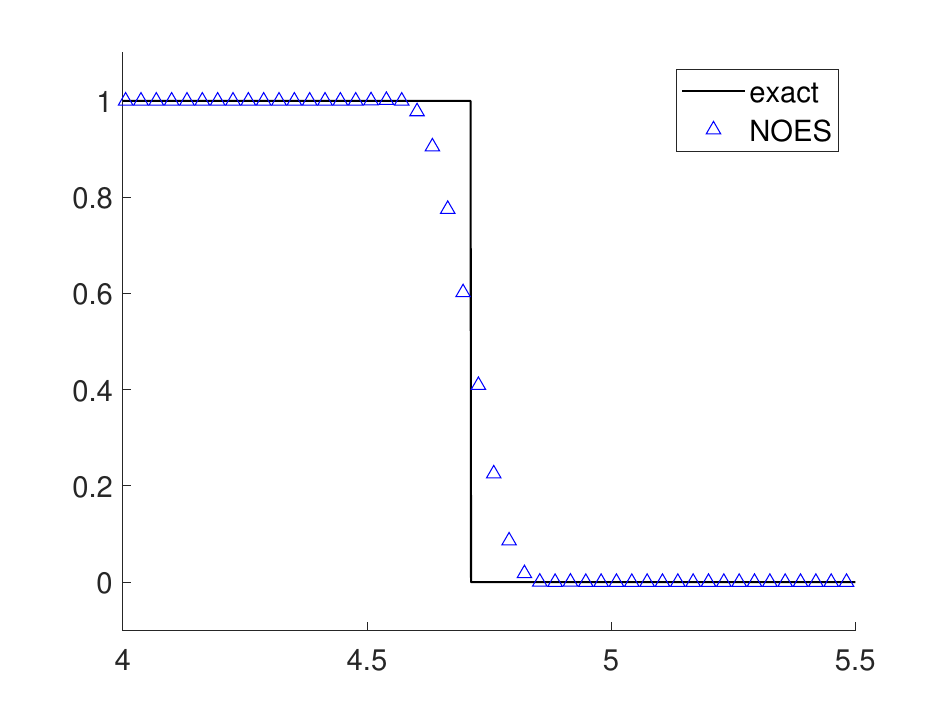}}
	\caption{Example \ref{ex:linear}: one dimensional linear equation with non-smooth initial data. The numerical solution at $T = 2\pi$ with $N = 200$.}
    \label{figstep}
\end{figure}

\begin{Ex}
\textbf{(Burgers' equation.)}
\label{ex:Burgers1D}
\end{Ex}

Here we consider the Burgers' equation
$$ u_t + \left(\frac{u^2}{2}\right)_x=0, \quad 
x\in \Omega = [0,2\pi]$$
with periodic boundary conditions. The initial condition is given by $u(x,0) = 0.5 + \sin x$.
We choose the entropy function as $U(u) = e^u$.

We first compare the numerical results at the final time $T = 0.6$, when the solution is still smooth. The errors and orders of accuracy are presented in Table \ref{tab2}. Again, we can obtain the optimal convergence rate for all cases. 

Moreover, in Figure \ref{figBurgers}, we present the results with $N=200$ cells at final time $T = 2.2$, when the solution contains a discontinuity in $\Omega$. We can see that our scheme captures the shock well without numerical oscillations. In Figure \ref{figU1}(b), we present the development of total entropy with time. It can be seen that the total entropy does not increasing.

\begin{table}[htb!]
        \centering
        \caption{Example \ref{ex:Burgers1D}: one-dimensional Burgers' equation. Errors and orders at final time $T = 0.6$.}
	\setlength{\tabcolsep}{3.2mm}{
		\begin{tabular}{|c|c|c|c|c|c|c|c|}
			\hline 
   & $N$ & $L^1$ error & order & $L^2$ error & order & $L^\infty$ error & order  \\ 
			\hline 
   \multirow{4}{*}{$k=1$}
    &64  &4.52e-04 & -- &8.12e-04 & -- &5.02e-03 & -- \\
   &128  &1.16e-04 & 1.97 &2.14e-04 &1.92 &1.38e-03 &1.86 \\
   &256  &2.95e-05 & 1.97 &5.56e-05 &1.95 &3.60e-04 &1.94 \\
   &512  &7.49e-06 & 1.98 &1.42e-05 &1.97 &9.14e-05 &1.98 \\
   \hline
   \multirow{4}{*}{$k=2$}
    &64  &1.22e-05 & --  &3.22e-05 & -- &2.82e-04 & -- \\
   &128  &1.45e-06 & 3.08 &3.84e-06 &3.07 &4.06e-05 &2.80 \\
   &256  &1.82e-07 & 2.99 &4.87e-07 &2.98 &5.44e-06 &2.90 \\
   &512  &2.28e-08 & 3.00 &6.17e-08 &2.98 &7.08e-07 &2.94 \\
			\hline 
   \multirow{4}{*}{$k=3$}
    &64  &4.28e-07 & -- &1.41e-06 & -- &8.69e-06 & -- \\
   &128  &2.42e-08 & 4.15 &8.30e-08 &4.08 &6.22e-07 &3.80 \\
   &256  &1.52e-09 & 3.99 &5.38e-09 &3.95 &4.10e-08 &3.92 \\
   &512  &9.69e-11 & 3.98 &3.46e-10 &3.96 &2.61e-09 &3.97 \\
   
			\hline 
	\end{tabular}} 
    \label{tab2}
\end{table}

\begin{figure}[htbp!]
	\centering
    \subfigure[Overview.]{
		\includegraphics[width=0.45\linewidth]{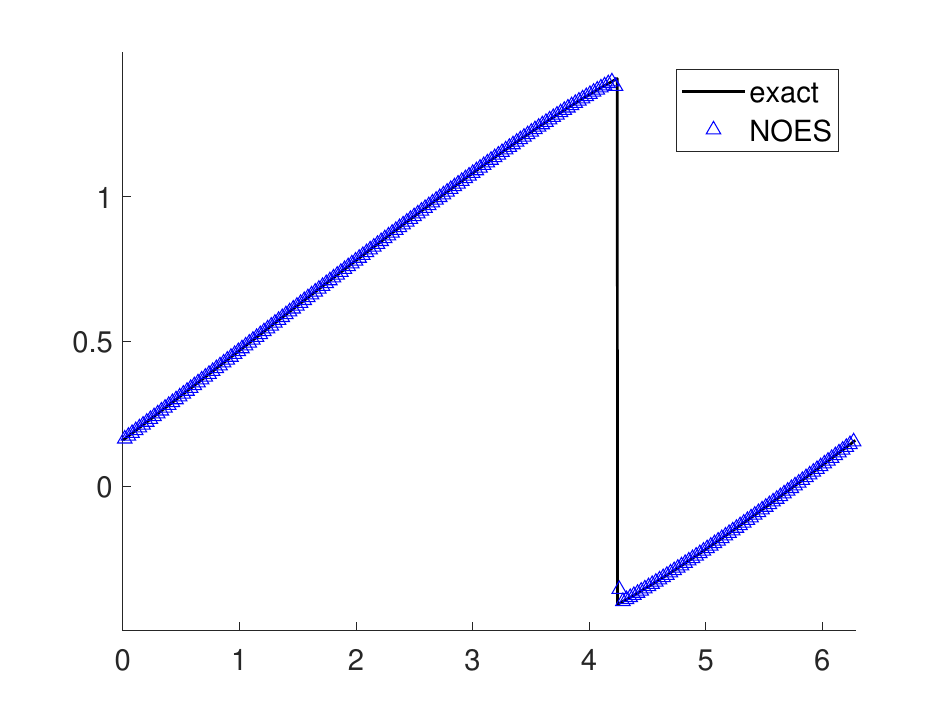}}
	\subfigure[Zoomed in results.]{
		\includegraphics[width=0.45\linewidth]{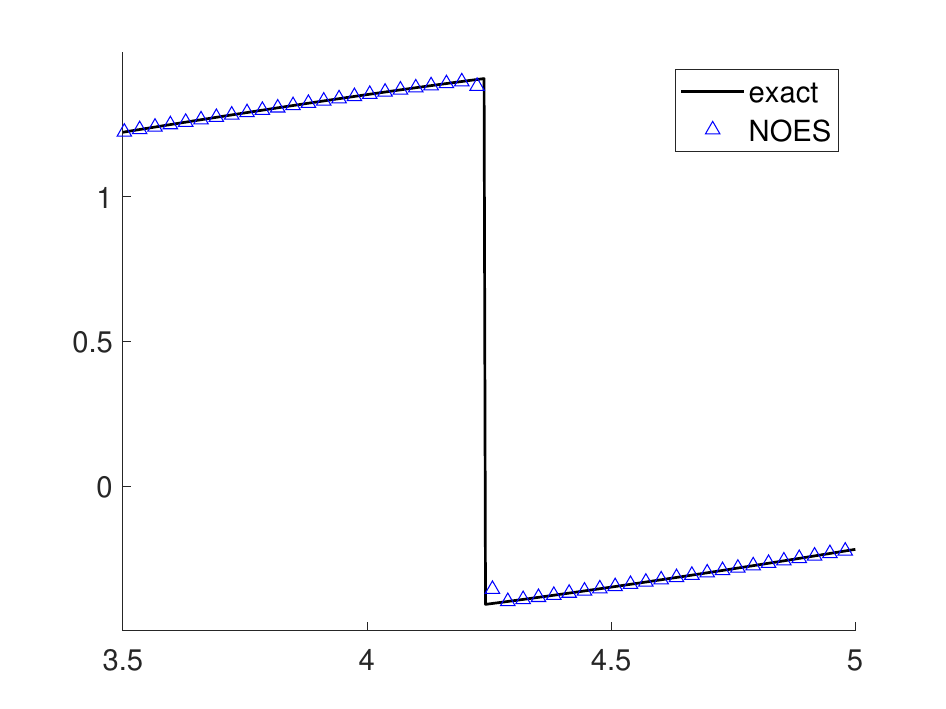}}
	\caption{Example \ref{ex:Burgers1D}: One-dimensional Burgers' equation. The numerical solution at $T = 2.2$ with $N=200$ cells.}
    \label{figBurgers}
\end{figure}

\begin{figure}[htbp!]
	\centering
    \subfigure[Linear equation.]{
		\includegraphics[width=0.45\linewidth]{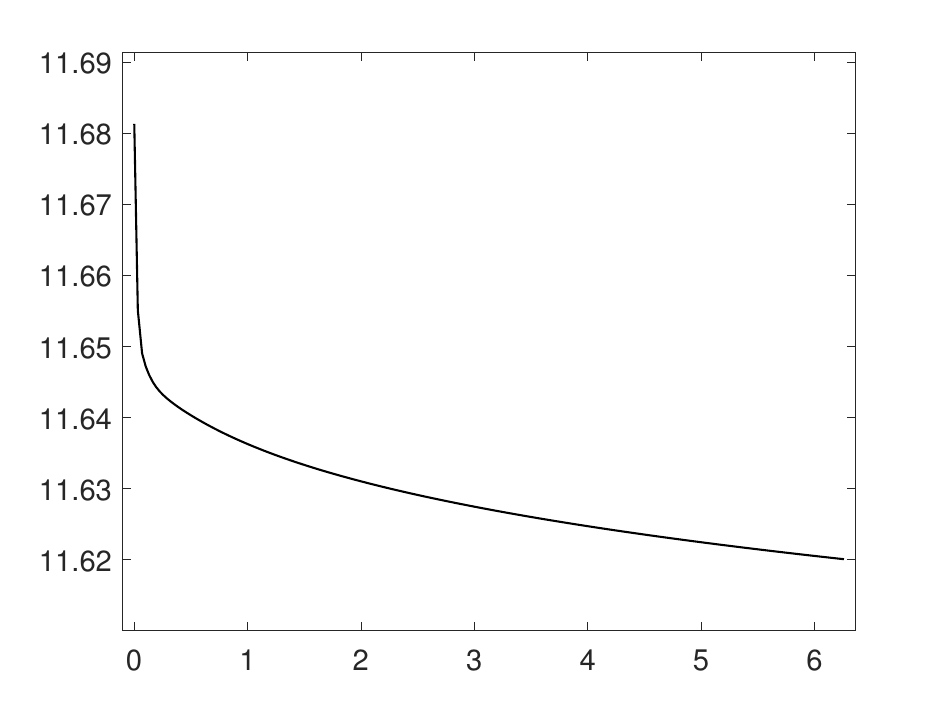}}
	\subfigure[Burgers' equation.]{
		\includegraphics[width=0.45\linewidth]{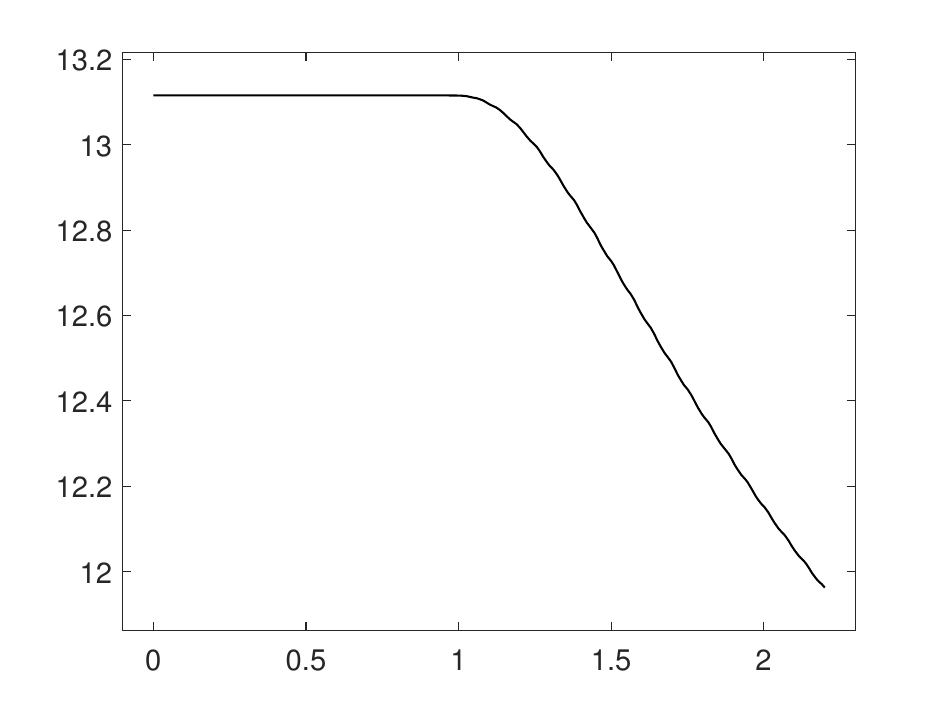}}
	\caption{One-dimensional tests: Evolution of the total entropy with time.}
    \label{figU1}
\end{figure}

\begin{Ex}
\textbf{(Buckley-Leverett equation.)}
\label{ex:BL1D}
\end{Ex}

In this example, we consider the one-dimensional Buckley-Leverett equation 
\begin{equation*}
    u_t + \left( \frac{4u^2}{4u^2+\left( 1-u \right) ^2} \right)_x = 0
\end{equation*}
with a Riemann initial condition
$$
u\left( x,0 \right) =\begin{cases}
	u_L,\quad x\le 0,\\
	u_R,\quad x>0.\\
\end{cases}
$$
This example is very challenging in computation because of the non-convexity of $f$. If
the numerical schemes are not carefully designed, they may fail to converge to the unique entropy solution or may be too slow to converge \cite{kurganov2007adaptive}. The computational domain is $\Omega=[-0.5,0.5]$. Following the setup in \cite{liu2024entropy}, we consider two kinds of initial conditions:
\begin{equation*}
\begin{aligned}
    \text{IC1:}\quad & u_L=-3, \, u_R=3, \\
    \text{IC2:}\quad & u_L=2, \, u_R=-2,
\end{aligned}
\end{equation*}
with three different entropy functions:
\begin{equation*}
\begin{aligned}
    \text{U1:}\quad &U(u) = \dfrac{u^2}{2}, \\
    \text{U2:}\quad &U(u) = \displaystyle\int \arctan(20u)\mathrm du,\\
    \text{U3:}\quad &U(u) = \displaystyle\int\arctan(u-1)\mathrm du.
\end{aligned}
\end{equation*}
The first entropy function is the standard square entropy, the second one is a
modified version of the Kruzhkov's entropy $U(u)=|u|$, and the third one is a
modified version of the translated Kruzhkov's entropy $U(u)=|u-1|$.

In Figure \ref{figBL1}, we present the results at $T = 1$ with $N= 200$ cells. The reference solution is simulated by the fifth-order finite difference WENO scheme on 10000 cells. 
From these results, we can see that for IC1, scheme with U2 gives the best simulation. Meanwhile, for IC2, only scheme with U3 can provide the satisfactory results. 
This is consistent with the results in \cite{chen2017entropy, liu2024entropy}. The results indicate that carefully selecting the entropy function helps to obtain physically correct solutions, especially for the nonconvex hyperbolic conservation laws. 

Moreover, to demonstrate the efficacy of parameter $\sigma$ in \eqref{scheme1}, we present the results for different settings for $\sigma$ in Figure \ref{figBL2}. We denote $\sigma_i=\sigma_i^{entropy}$ as "entropy",  $\sigma_i=\sigma_i^{jump}$ as "jump", and the proposed NOES-DG scheme as "entropy + jump" in following. It can be seen that even for IC1 with U2, the solution is incorrect if we only concern the entropy term. And for IC2 with U3, the solution is incorrect when we only consider the jump term. These results show that two parameters, $\sigma_i^{jump}$ and $\sigma_i^{entropy}$, are both necessary.

\begin{figure}[htbp!]
	\centering
    \subfigure[IC1.]{
		\includegraphics[width=0.45\linewidth]{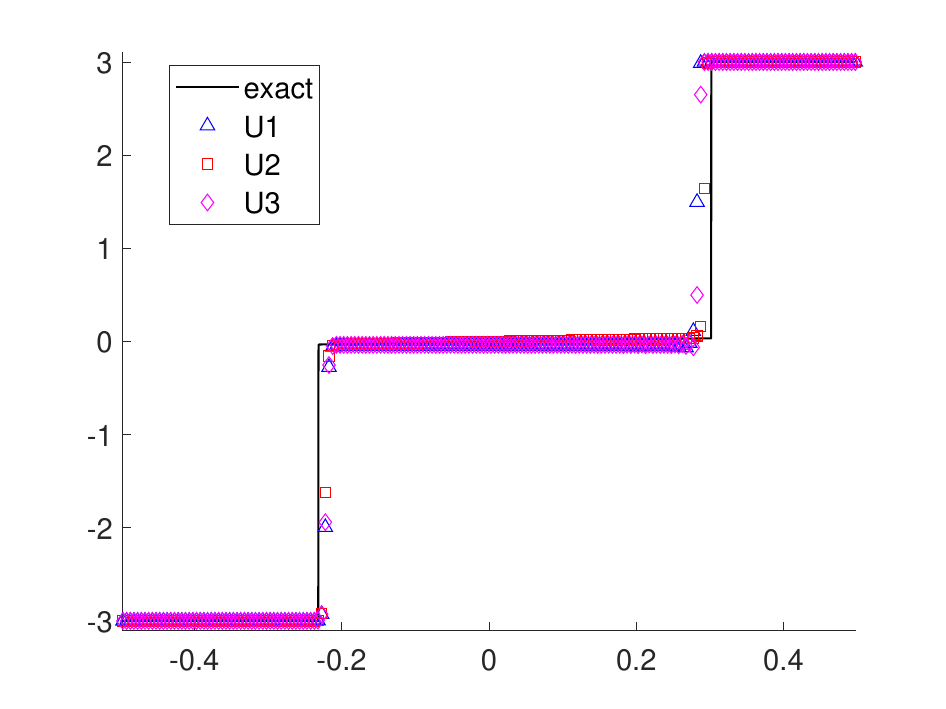}}
	\subfigure[Zoomed in (a).]{
		\includegraphics[width=0.45\linewidth]{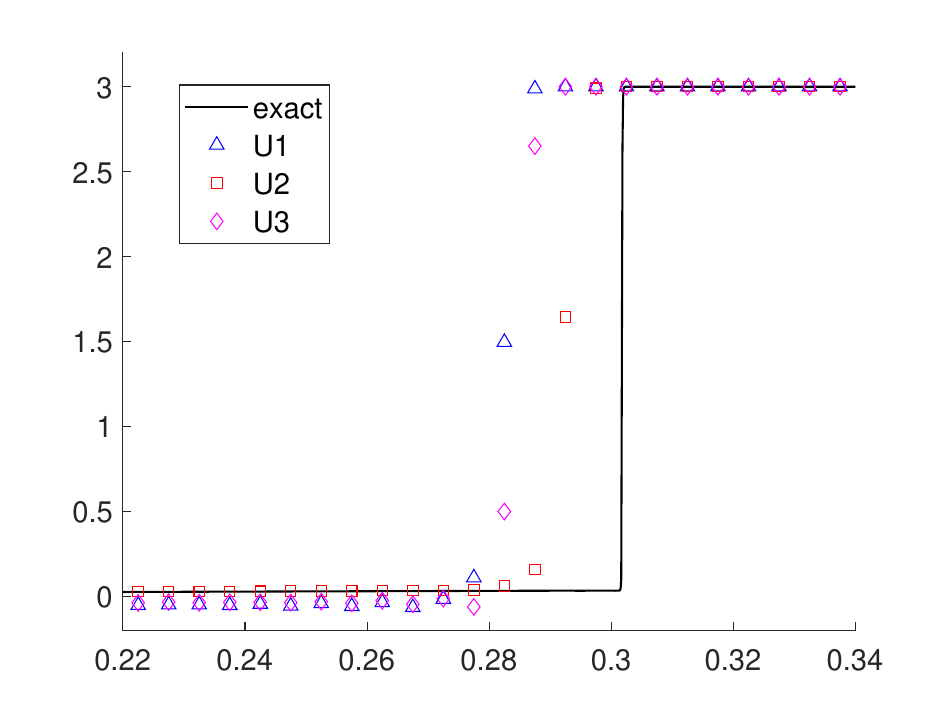}}
    \subfigure[IC2.]{
		\includegraphics[width=0.45\linewidth]{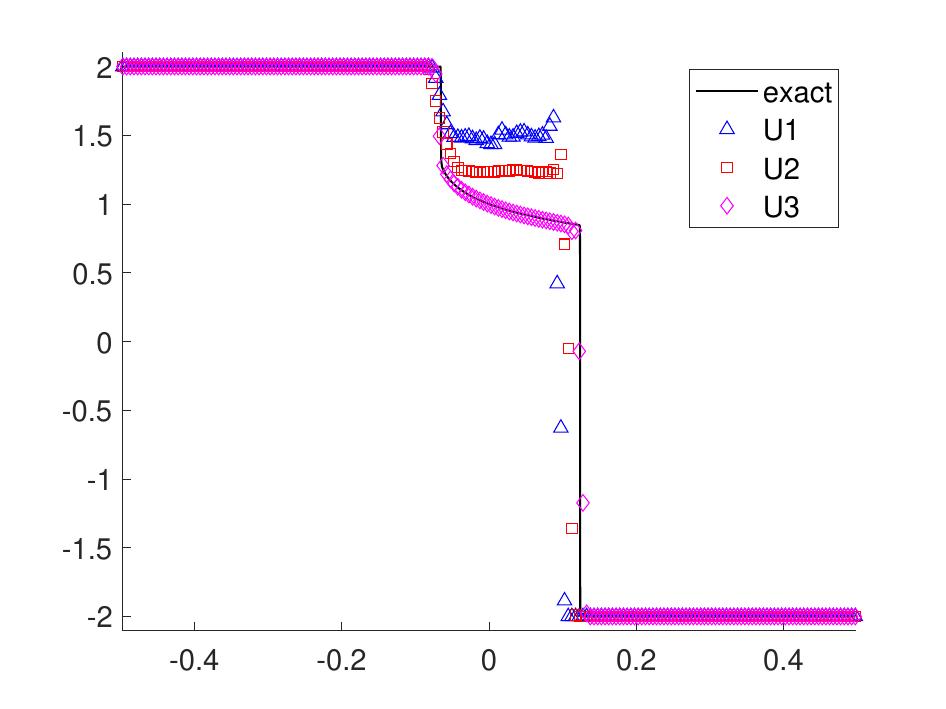}}
	\subfigure[Zoomed in (c).]{
		\includegraphics[width=0.45\linewidth]{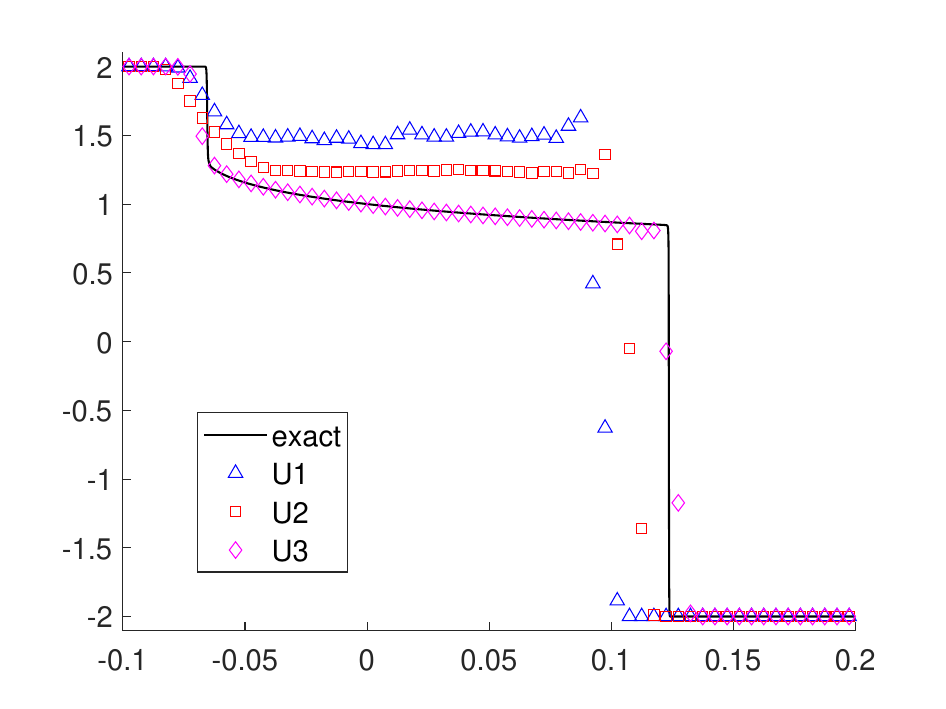}}
	\caption{Example \ref{ex:BL1D}: One dimensional Buckley-Leverett problem with different initial conditions. Numerical solutions at $T = 1$ with $N= 200$ cells.}
 \label{figBL1}
\end{figure}

\begin{figure}[htbp!]
	\centering
    \subfigure[IC1 with U2.]{
		\includegraphics[width=0.45\linewidth]{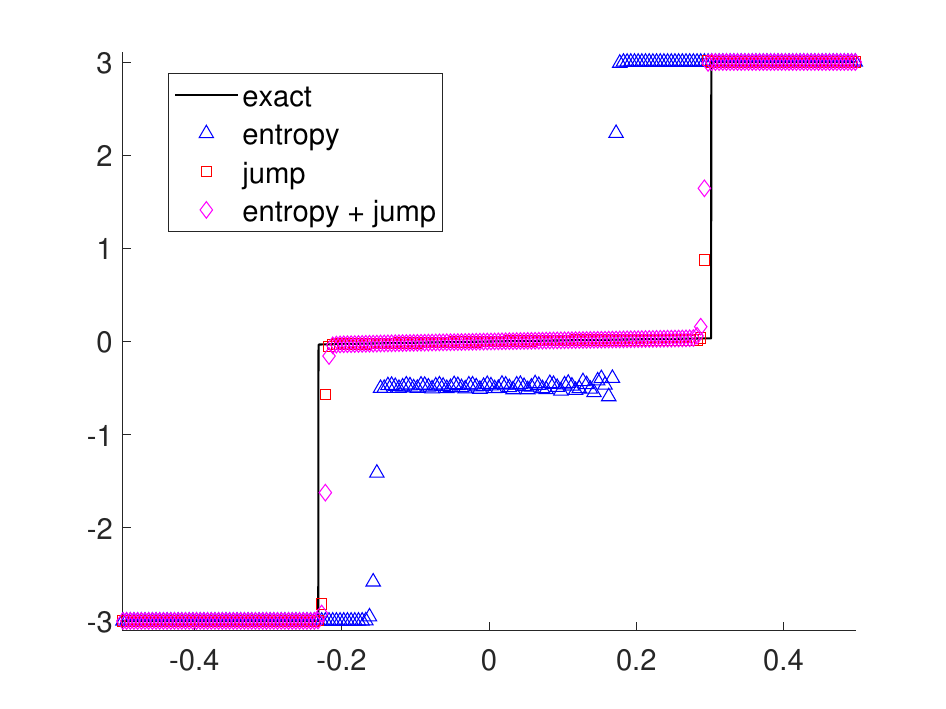}}
	\subfigure[IC2 with U3.]{
		\includegraphics[width=0.45\linewidth]{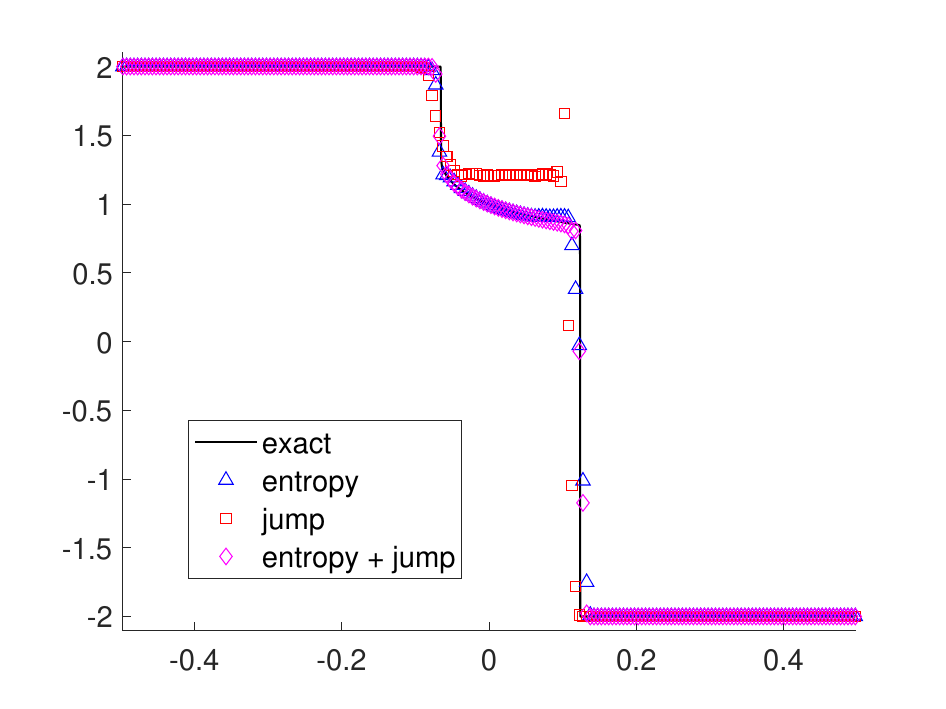}}
	\caption{Example \ref{ex:BL1D}: One dimensional Buckley-Leverett problem with different settings of $\sigma$. Numerical solutions at $T = 1$ with $N= 200$ cells.}
 \label{figBL2}
\end{figure}

\subsection{Two-dimensional tests}

\begin{Ex}
\textbf{(Burgers' equation.)}
\label{ex:Burgers2D}
\end{Ex}

Here we consider the two-dimensional Burgers' equation
$$ u_t + \left(\frac{u^2}{2}\right)_x + \left(\frac{u^2}{2}\right)_y=0.$$
The computational domain is $\Omega = [0,2\pi]\times[0,2\pi]$ with periodic boundary conditions in each direction. The initial condition is given as
$$ u(x,y,0) = \sin (x + y). $$
In this example, we choose the entropy function as $U(u) = \cosh(u)$.

At first, we perform accuracy test at the final time $T = 0.3$, when the solution is still smooth. 
Errors and orders are presented in Table \ref{tab3}. We can clearly observe
that the scheme achieves the designed $(k+1)$-th order accuracy. 
Next, in Figure \ref{figBurgers2D}, we plot the numerical solution with $N_x = N_y = 200$ at final time $T = 0.8$, when shocks have already appeared. A one-dimensional cut at $x = y$ of
the solution is also presented. We can see that the shocks are well resolved without spurious oscillations. In Figure \ref{figU2}(a), we present the evolution of total entropy with time. It can be seen that the entropy does not increasing.

\begin{table}[htb!]
        \centering
        \caption{Example \ref{ex:Burgers2D}: Two-dimensional Burgers' equation. Errors and orders at final time $T = 0.3$.}
	\setlength{\tabcolsep}{3.2mm}{
		\begin{tabular}{|c|c|cc|cc|cc|}
			\hline 
   & $N_x\times N_y$ & $L^1$ error & order & $L^2$ error & order & $L^\infty$ error & order  \\ \hline 
   \multirow{4}{*}{$k=1$}
   &$64\times 64$ & 1.85E-03
			& -- & 4.95E-03 & -- & 4.78E-02 & --   \\ 
   &$128\times 128$ & 2.47E-04 & 2.91 
			& 5.12E-04 & 3.27 
			& 4.86E-03 & 3.30 		
			\\ 
   &$256\times 256$ & 5.07E-05 & 2.28 
			& 1.04E-04
			& 2.30
			& 7.49E-04
			& 2.70		
			\\ 
   &$512\times 512$ & 1.27E-05
			& 1.99
			& 2.63E-05
			& 1.99 
			& 1.90E-04
			& 1.98		
			\\ 
			\hline
   \multirow{4}{*}{$k=2$}
   &$64\times 64$ & 4.73E-05
			& -- & 1.08E-04 & -- & 1.04E-03 &  --  \\ 
			&$128\times 128$ & 5.83E-06 & 2.91 
			& 1.47E-05 & 2.88 
			& 1.23E-04 & 3.08
			\\ 
			&$256\times 256$ & 7.59E-07 & 2.94
			& 2.00E-06
			& 2.88
			& 2.06E-05
			& 2.59
			\\ 
			& $512\times 512$ & 9.78E-08
			& 2.96
			& 2.69E-07
			& 2.90
			& 3.12E-06
			& 2.72
			\\ 
			\hline 
    \multirow{4}{*}{$k=3$}			
			&$64\times 64$ & 3.64E-06
			& -- & 1.25E-05 & -- & 1.50E-04 & --   \\ 
			&$128\times 128$ & 2.13E-07 & 4.09 
			& 7.59E-07 & 4.04 
			& 1.07E-05 & 3.82 
			\\ 
			& $256\times 256$ & 1.35E-08 & 3.98 
			& 4.67E-08
			& 4.02
			& 5.49E-07
			& 4.28
			\\ 
			& $512\times 512$ & 8.53E-10
			& 3.98
			& 2.96E-09
			& 3.98
			& 3.48E-08
			& 3.98
			\\ 
			\hline 
	\end{tabular}} 
    \label{tab3}
\end{table}

\begin{figure}[htbp!]
	\centering
    \subfigure[Contour.]{
		\includegraphics[width=0.45\linewidth]{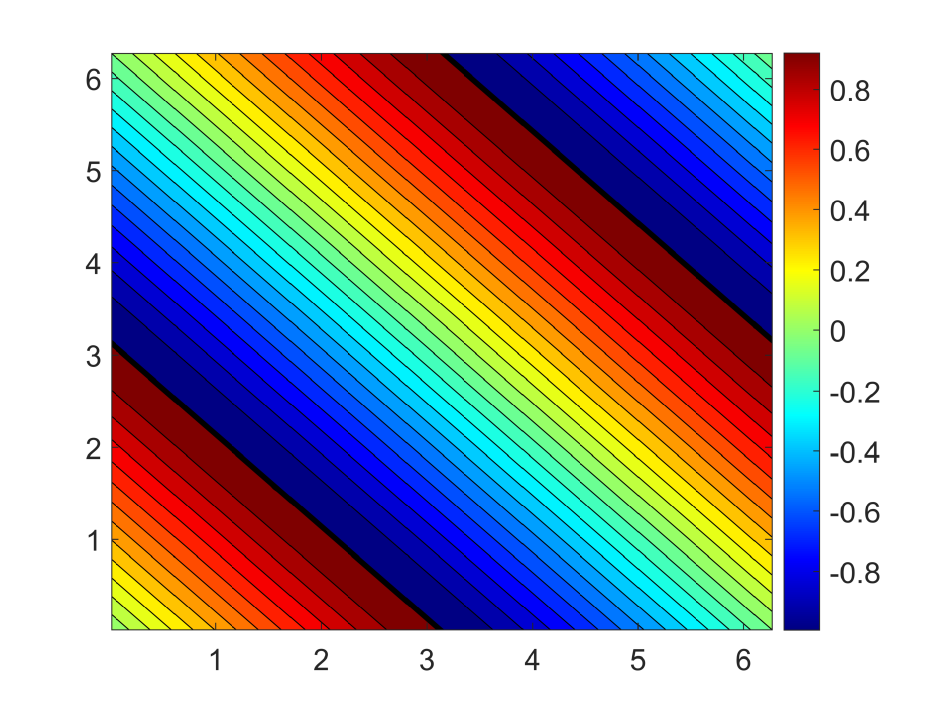}}
	\subfigure[Cut at $x = y$.]{
		\includegraphics[width=0.45\linewidth]{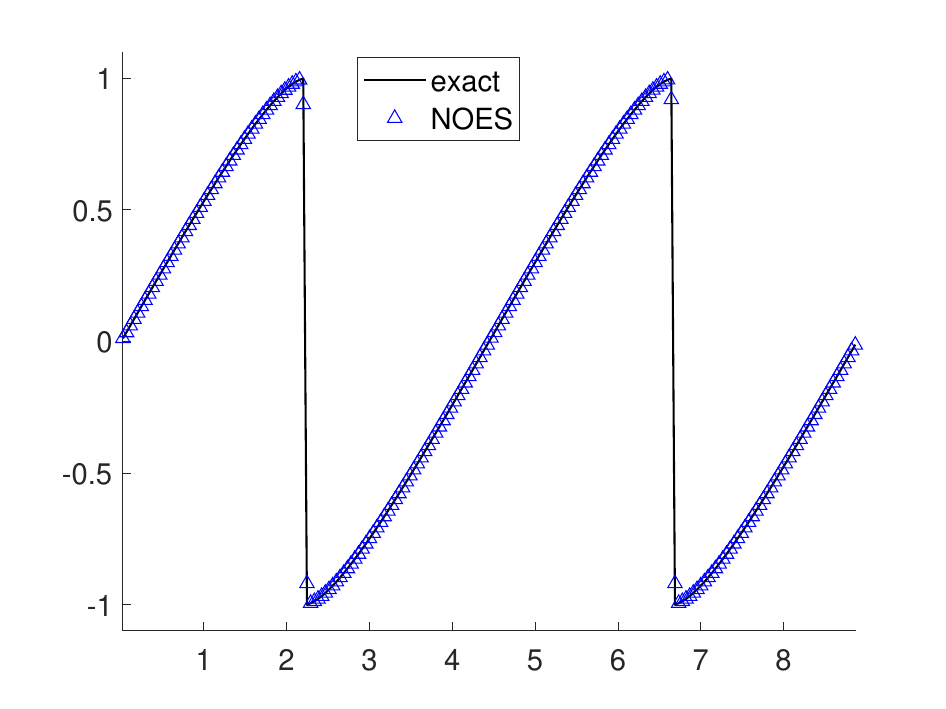}}
	\caption{Example \ref{ex:Burgers2D}: Two-dimensional Burgers' equation. The numerical solution at $T = 0.8$ with $N_x = N_y=200$.}
 \label{figBurgers2D}
\end{figure}

\begin{Ex}
\textbf{(KPP problem.)}
\label{ex:KPP}
\end{Ex}

In this example, we consider the two-dimensional KPP problem \cite{kurganov2007adaptive}
$$ u_t+\left(\sin u\right)_x + \left(\cos u\right)_y=0, $$
which is very challenging for many high-order numerical schemes because the solution has a
two-dimensional composite wave structure. The computational domain is taken as $\Omega=[-1.5,1.5]\times[-2,1]$ with periodic boundary condition in each direction. 
The initial data is given as
$$
u\left( x,y,0 \right) =\begin{cases}
	3.5\pi , & x^2+y^2<0.5,\\
	0.25\pi ,& \mathrm{otherwise}.\\
\end{cases}
$$
In this example, we take the entropy function as $U=\cosh(u)$. In Figure \ref{figKPP}, we present the result at $T = 1$ with $200\times 200$ meshes. It can be seen that our result agrees well with that in \cite{kurganov2007adaptive}. In Figure \ref{figU2}(b), we present the evolution with total entropy, and it is not increasing with time.  

\begin{figure}[htbp]
    \centering
	\includegraphics[width=0.5\linewidth]{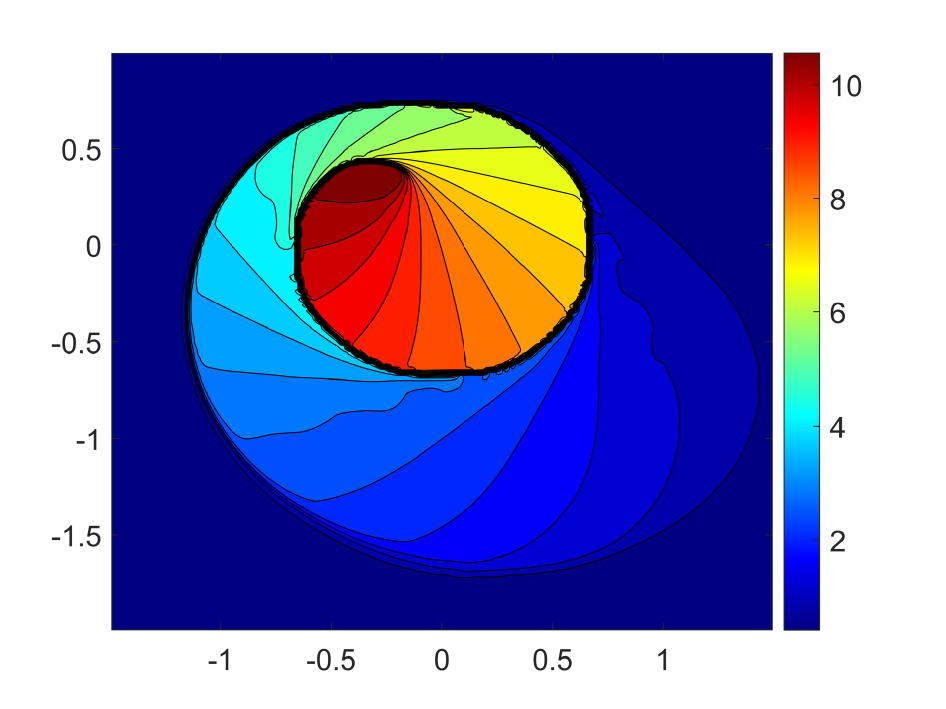}
    \caption{Example \ref{ex:KPP}: KPP problem. The numerical solution at $T = 1$ with $N_x\times N_y = 200\times 200$.}
    \label{figKPP}
\end{figure}

\begin{figure}[htbp!]
	\centering
    \subfigure[Burgers' equation.]{
		\includegraphics[width=0.45\linewidth]{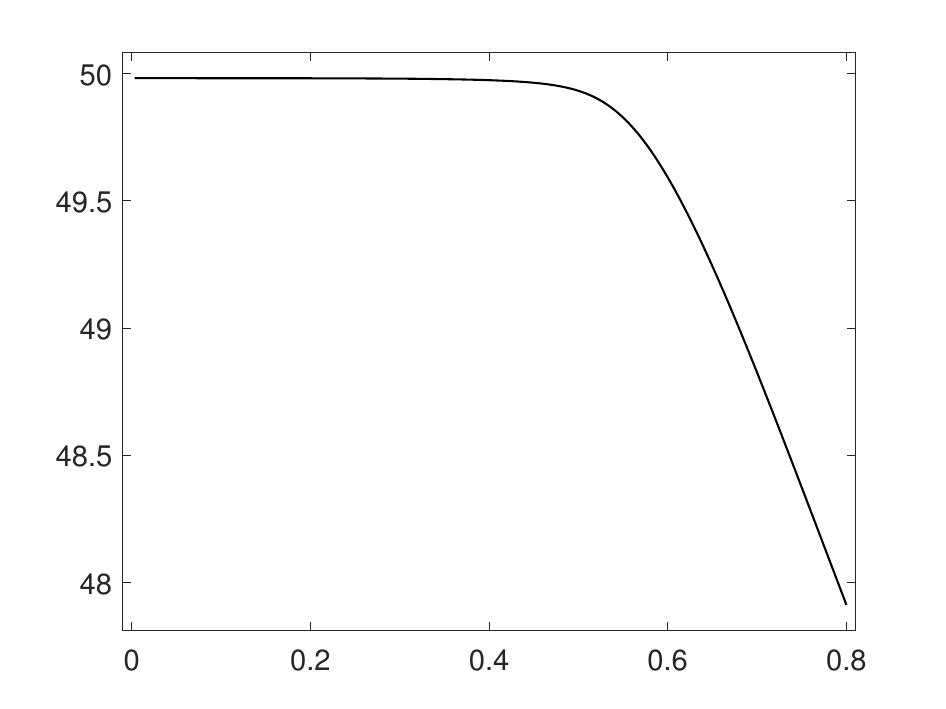}}
	\subfigure[KPP problem.]{
		\includegraphics[width=0.45\linewidth]{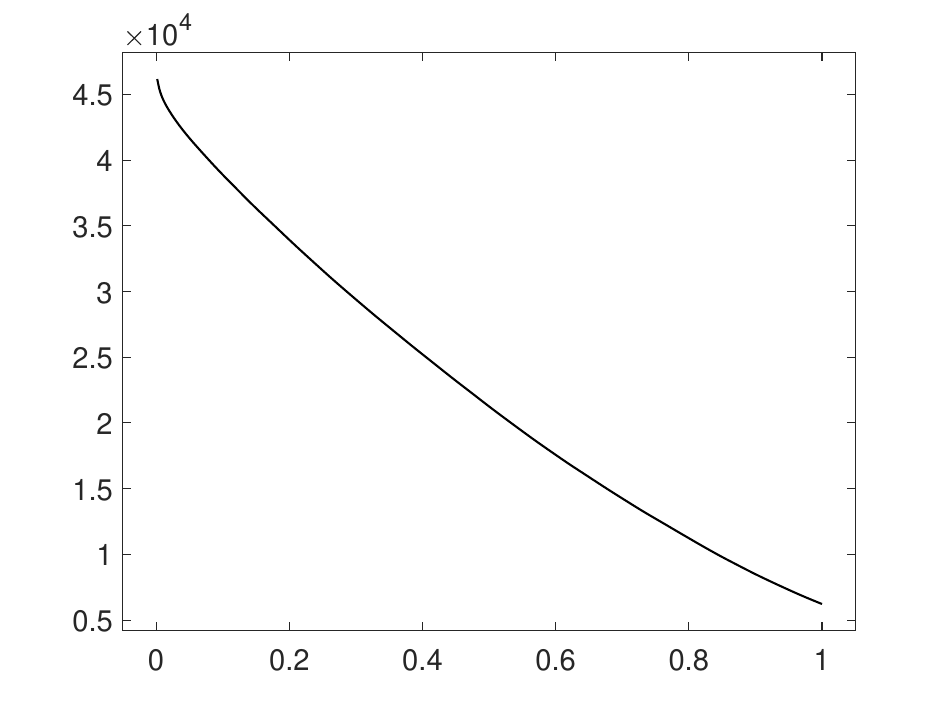}}
	\caption{Two-dimensional tests: The evolution of total entropy with time.}
    \label{figU2}
\end{figure}

\subsection{Euler system}

For Euler system, we employ the HLL flux in the simulations.  The PP limiter is utilized on each stage of the RK method used.

\begin{Ex}
\textbf{(1D accuracy test.)}
\label{ex:Eulersin}
\end{Ex}

At first, we test the accuracy of the scheme on one-dimensional Euler equation. The computational domain is $\Omega = [0,2\pi]$ with periodic boundary condition, the initial data is given by
$$(\rho,u,p)|_{t=0} = (1+0.2\sin x,1,1).$$
The exact solution is
$$ (\rho,u,p) = (1 + 0.2\sin(x-t),1,1). $$
In Table \ref{tab4}, we present the errors and orders of $\rho$ at $T = 1$. Clearly, the optimal convergence rate is obtained.

\begin{table}[htb!]
        \centering
        \caption{Example \ref{ex:Eulersin}: One-dimensional Euler equation. Errors and orders of density $\rho$ with the final time $T = 1$.}
	\setlength{\tabcolsep}{3.2mm}{
		\begin{tabular}{|c|c|cc|cc|cc|}
			\hline 
   &$N$ & $L^1$ error & order & $L^2$ error & order & $L^\infty$ error & order  \\ 
			\hline 
   \multirow{4}{*}{$k=1$}
    &64  &6.33e-05 & -- &8.29e-05 & -- &2.59e-04 & -- \\ 
   &128  &1.58e-05 & 2.01 &2.07e-05 &2.00 &6.53e-05 &1.99 \\ 
   &256  &3.94e-06 & 2.00 &5.18e-06 &2.00 &1.64e-05 &1.99 \\
   &512  &9.83e-07 & 2.00 &1.30e-06 &2.00 &4.11e-06 &2.00 \\
			\hline  
   \multirow{4}{*}{$k=2$}
    &64  &5.23e-07 & -- &6.70e-07 & -- &2.75e-06 & -- \\ 
   &128  &6.48e-08 & 3.01 &8.19e-08 &3.03 &2.60e-07 &3.40 \\
   &256  &8.09e-09 & 3.00 &1.02e-08 &3.00 &3.06e-08 &3.09 \\ 
   &512  &1.01e-09 & 3.00 &1.28e-09 &3.00 &3.80e-09 &3.01 \\
			\hline 
   \multirow{4}{*}{$k=3$}
    &64  &3.15e-09 & -- &4.06e-09 & -- &1.21e-08 & -- \\ 
   &128  &1.93e-10 & 4.03 &2.48e-10 &4.03 &6.01e-10 &4.33 \\ 
   &256  &1.20e-11 & 4.00 &1.55e-11 &4.00 &3.48e-11 &4.11 \\ 
   &512  &7.57e-13 & 3.99 &9.71e-13 &3.99 &2.27e-12 &3.94 \\
			\hline 
	\end{tabular}} 
    \label{tab4}
\end{table}

\begin{Ex}
\textbf{(Sod shock tube.)}
\label{ex:Sod}
\end{Ex}

Here, we test the 1D Sod shock tube problem \cite{sod1978survey}, which is a standard example for Euler equation. The computational domain is taken as $\Omega = [0,1]$, and the initial data is given by
$$
\left( \rho ,u,p \right) =\begin{cases}
	\left( 1,0,1 \right) , & x<0.5,\\
	\left( 0.125,0,0.1 \right) , & x\ge 0.5.\\
\end{cases}
$$
In Figure \ref{figSod}, we compare the result of density at $T = 0.2$ with exact solution on $N=200$ meshes. We observe that the results are well-resolved and free of oscillations.
To examine the effect of $c_0$, we present results for different values  $c_0=0.2,\ 1,\ 5$ in Figure \ref{figSodsigma}. It is observed that the performance is not highly sensitive to $c_0$.

\begin{figure}[htbp!]
	\centering
    \subfigure[Density.]{
		\includegraphics[width=0.45\linewidth]{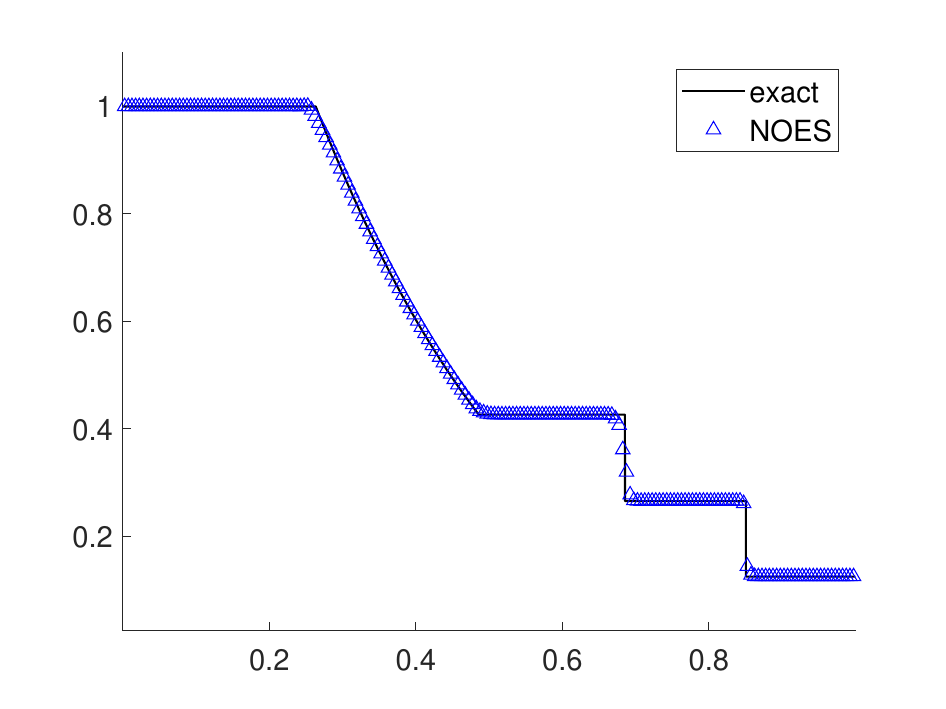}}
	\subfigure[Zoomed in results.]{
		\includegraphics[width=0.45\linewidth]{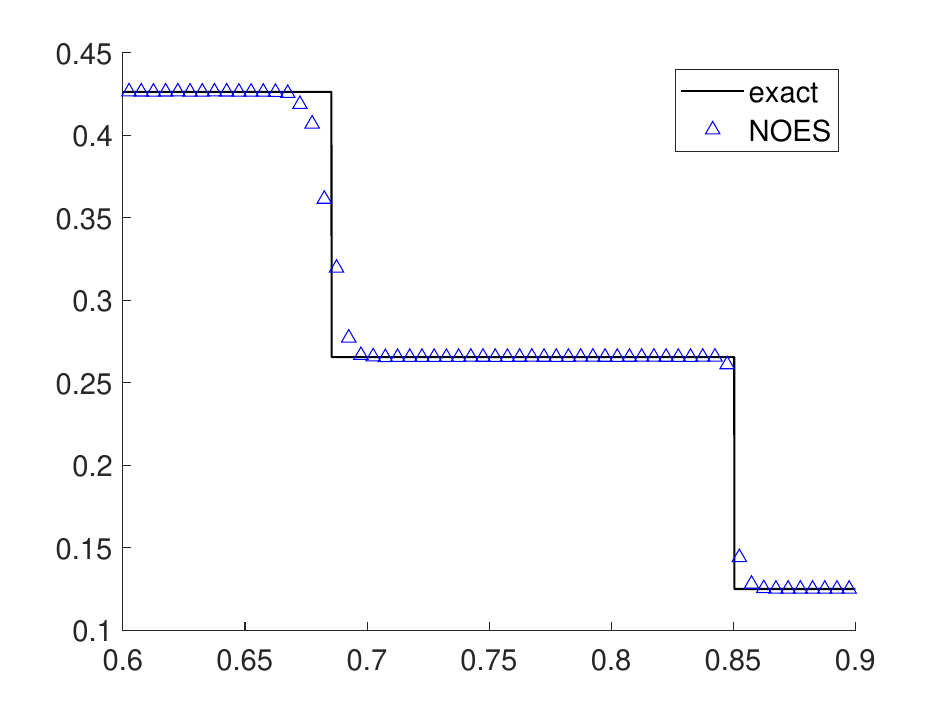}}
	\caption{Example \ref{ex:Sod}: Sod shock tube. The numerical solution at $T = 0.2$ with $N = 200$.}
    \label{figSod}
\end{figure}

\begin{figure}[htbp!]
	\centering
    \subfigure[Density with different $c_0$.]{
		\includegraphics[width=0.45\linewidth]{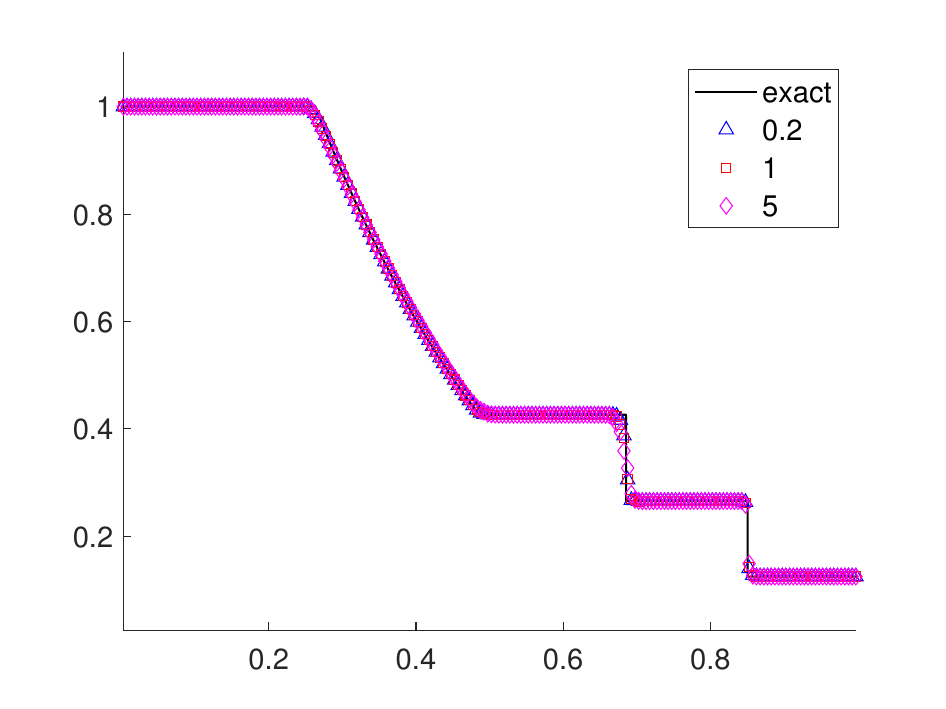}}
	\subfigure[Zoomed in results.]{
		\includegraphics[width=0.45\linewidth]{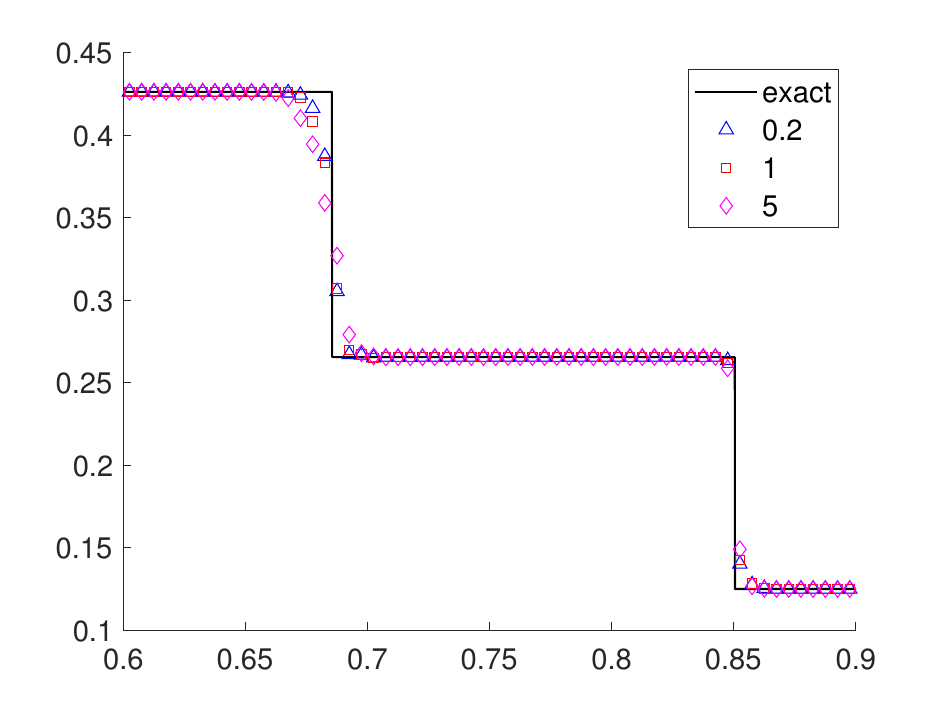}}
	\caption{Example \ref{ex:Sod}: Sod shock tube. The numerical solution at $T = 0.2$ with $N = 200$ with different values of $c_0$.}
    \label{figSodsigma}
\end{figure}

\begin{Ex}
\textbf{(Shu-Osher problem.)}
\label{ex:Shu}
\end{Ex}

Next, we test the Shu-Osher problem \cite{shu1988efficient}. The solution has complicated structure in that it contains both strong and weak shock waves and highly oscillatory smooth waves. The computational domain is taken as $\Omega = [-5,5]$. And the initial data is 
$$
\left( \rho ,u,p \right) =\begin{cases}
	\left( 3.857143,2.629369,10.3333 \right) , \quad x<-4,\\
	\left( 1+0.2\sin \left( 5x \right) ,0,1 \right) , \qquad x\ge 4.\\
\end{cases}
$$
In Figure \ref{figShu}, we present the results of density and pressure at $T = 1.8$ with $N=200$.
 The reference solution is computed by the fifth-order finite difference WENO scheme with 2000 cells. Our numerical results seem to agree with the reference solution very well.

\begin{figure}[htbp!]
	\centering
    \subfigure[Density.]{
		\includegraphics[width=0.45\linewidth]{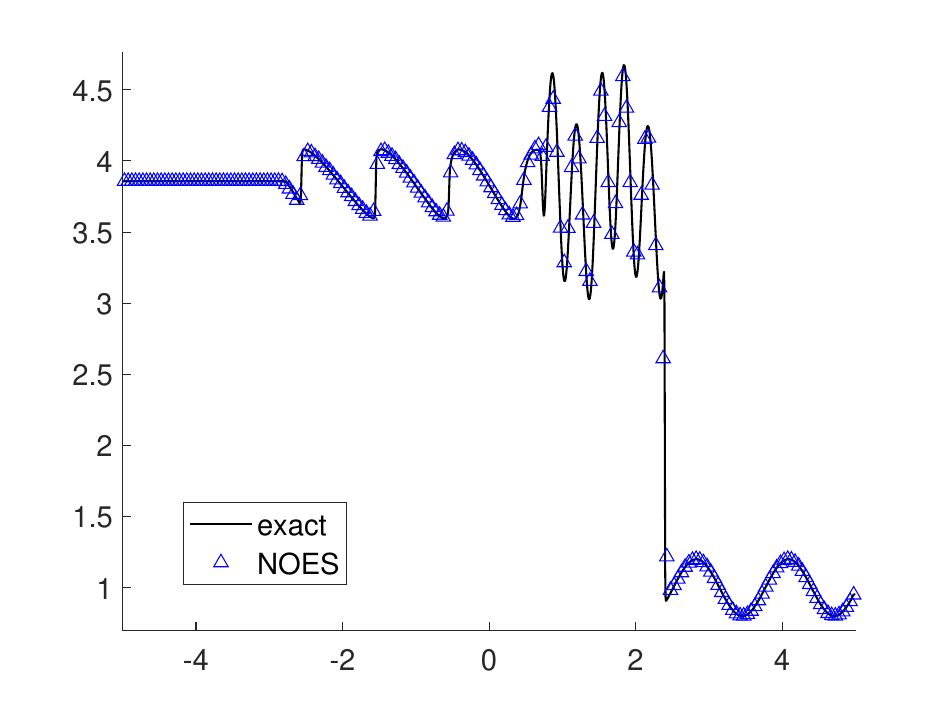}}
	\subfigure[Zoomed in results.]{
		\includegraphics[width=0.45\linewidth]{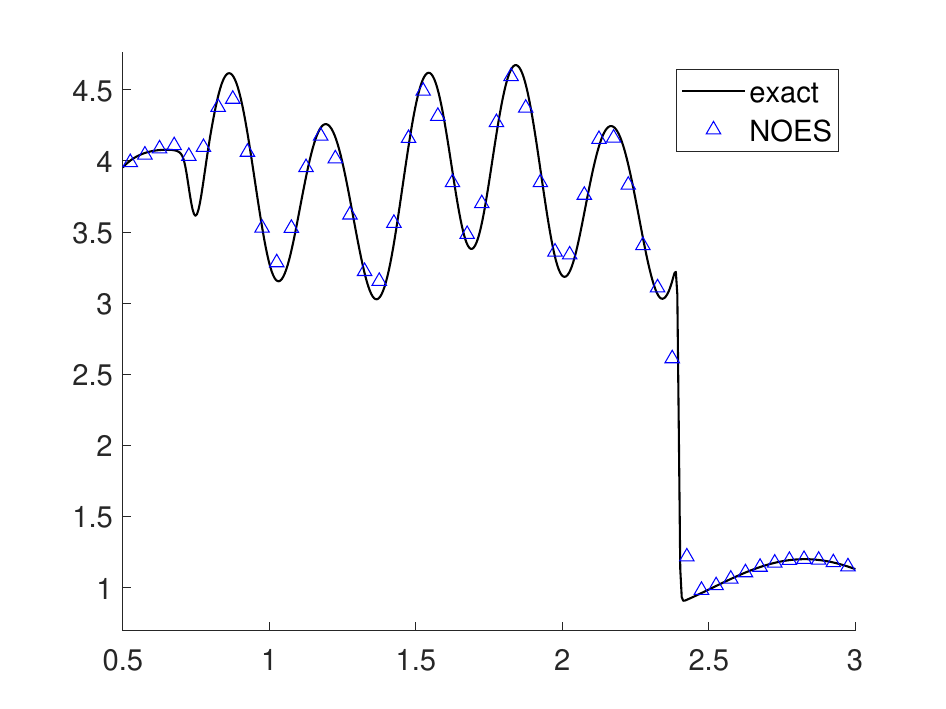}}
	\caption{Example \ref{ex:Shu}: Shu-Osher problem. The density of numerical solution at $T = 1.8$ with $N= 200$.}
    \label{figShu}
\end{figure}

\begin{Ex}
\textbf{(Two blast waves.)}
\label{ex:Blast}
\end{Ex}

Here, we consider the two blast-waves problem \cite{woodward1984numerical},  which presents a significant challenge for most high-order schemes. The computational domain is $\Omega = [0,1]$ with reflective boundary conditions imposed on both end points, and the initial data is given by
$$
\rho =1,\quad u=0,\quad p=\begin{cases}
	1000, & x<0.1,\\
	0.01,& 0.1\le x\le 0.9,\\
	100,& x>0.9.\\
\end{cases}
$$
The PP limiter is needed to ensure positivity of pressure and density. In Figure \ref{figBlast}, we present the result of density at $T = 0.038$ with $N= 800$ cells. The reference solution is  obtained with the fifth-order WENO scheme with 2000 cells.  The proposed scheme produces high-quality numerical results with sharp, non-oscillatory shock transitions.

\begin{figure}[htbp!]
	\centering
    \subfigure[Density.]{
		\includegraphics[width=0.45\linewidth]{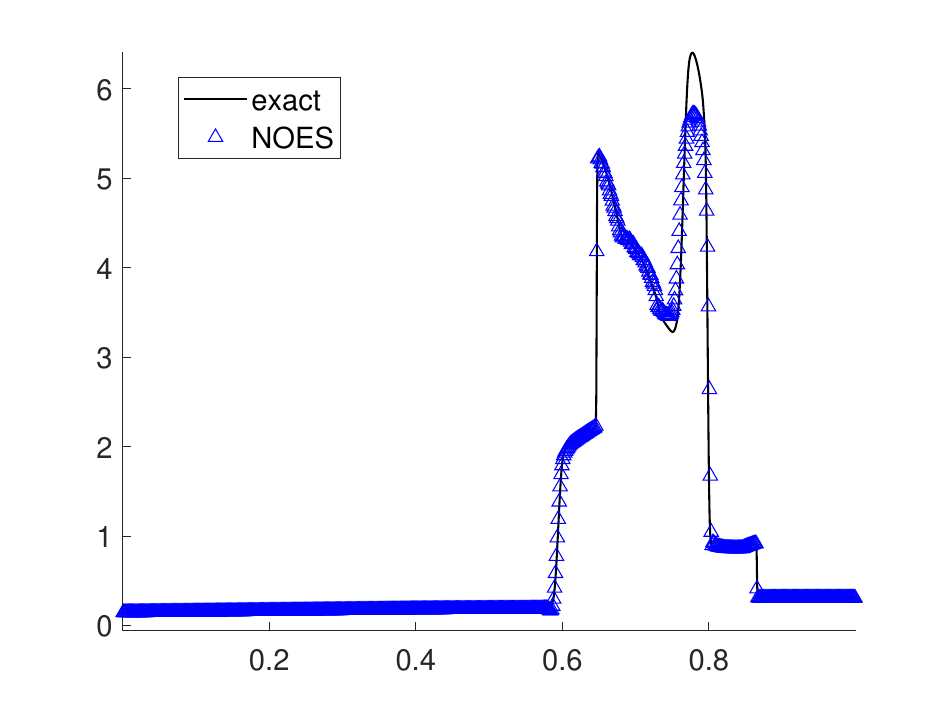}}
	\subfigure[Zoomed in results.]{
		\includegraphics[width=0.45\linewidth]{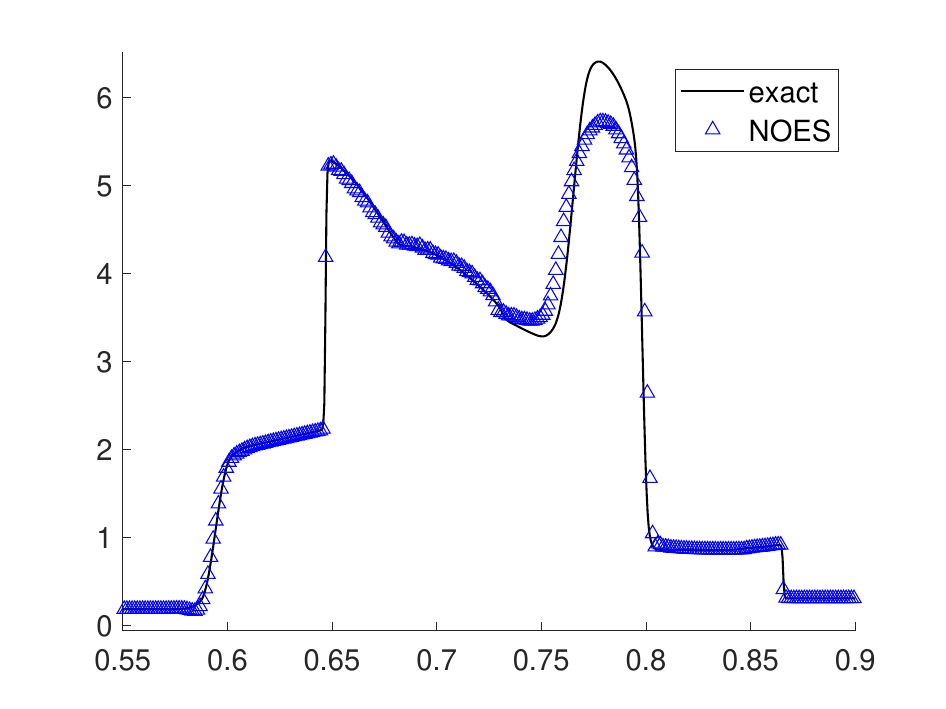}} 
	\caption{Example \ref{ex:Blast}: Two blast waves. The density of numerical solution at $T = 0.038$ with $N_x = 800$.}
    \label{figBlast}
\end{figure}

\begin{Ex}
\textbf{(Leblanc shock tube.)}
\label{ex:Leblanc}
\end{Ex}

Here, we consider the Leblanc problem \cite{loubere2005subcell}, which is an extreme Riemann problem with very strong discontinuities and is used to test robustness of our scheme. The computational domain is $\Omega = [-10,10]$, and the initial data is given by
$$
\left( \rho ,u,p \right) =\begin{cases}
	\left( 2,0,10^9 \right) , & x<0,\\
	\left( 0.001,0,1 \right) , & x\ge 0.\\
\end{cases}
$$
This example, like the previous two blast waves test, requires the PP limiter.  In Figure \ref{figLeblanc}, we compare the result of density, velocity and pressure with the exact solution at $T = 0.0001$ on $N = 800$ and $N = 6400$ meshes. 
 These results show the robustness and convergence of our scheme.

\begin{figure}[htbp!]
	\centering
    \subfigure[$\lg\rho$.]{
		\includegraphics[width=0.45\linewidth]{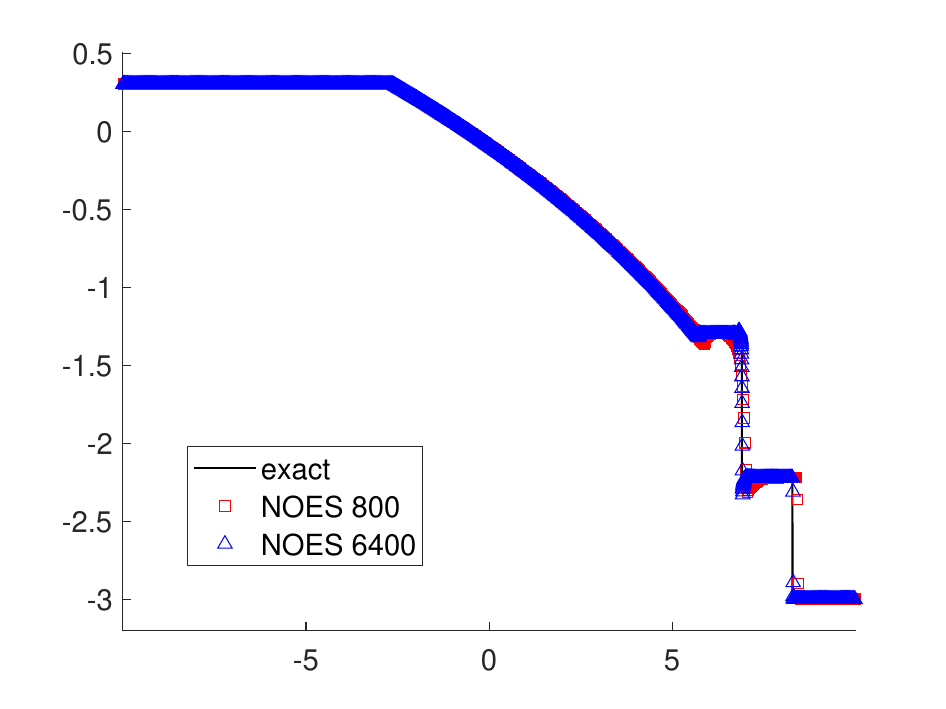}}
	\subfigure[Zoomed in (a).]{
		\includegraphics[width=0.45\linewidth]{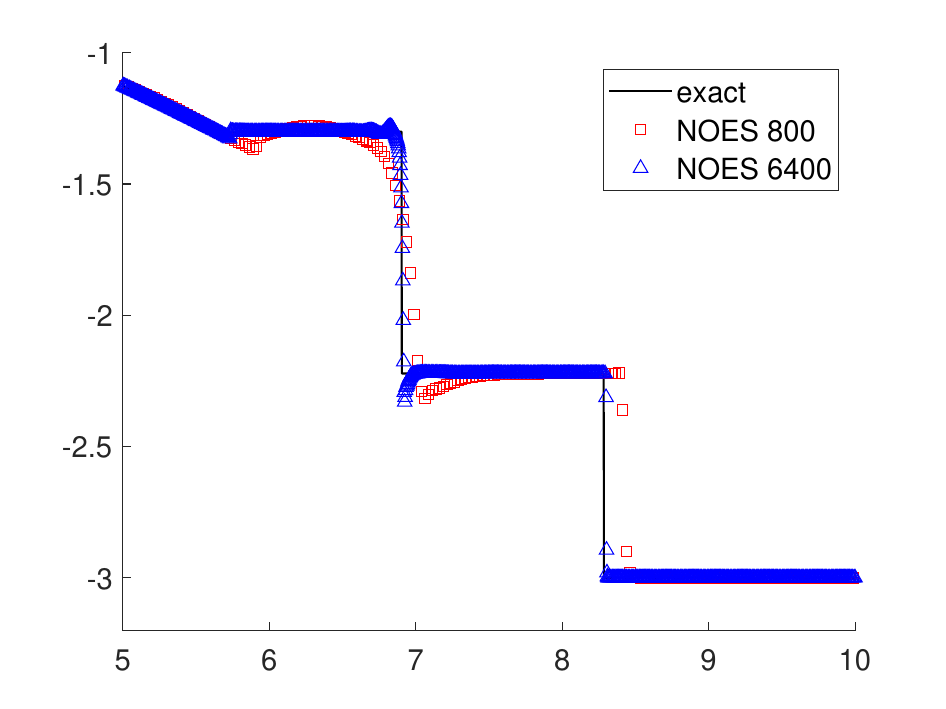}}
    \subfigure[Velocity.]{
		\includegraphics[width=0.45\linewidth]{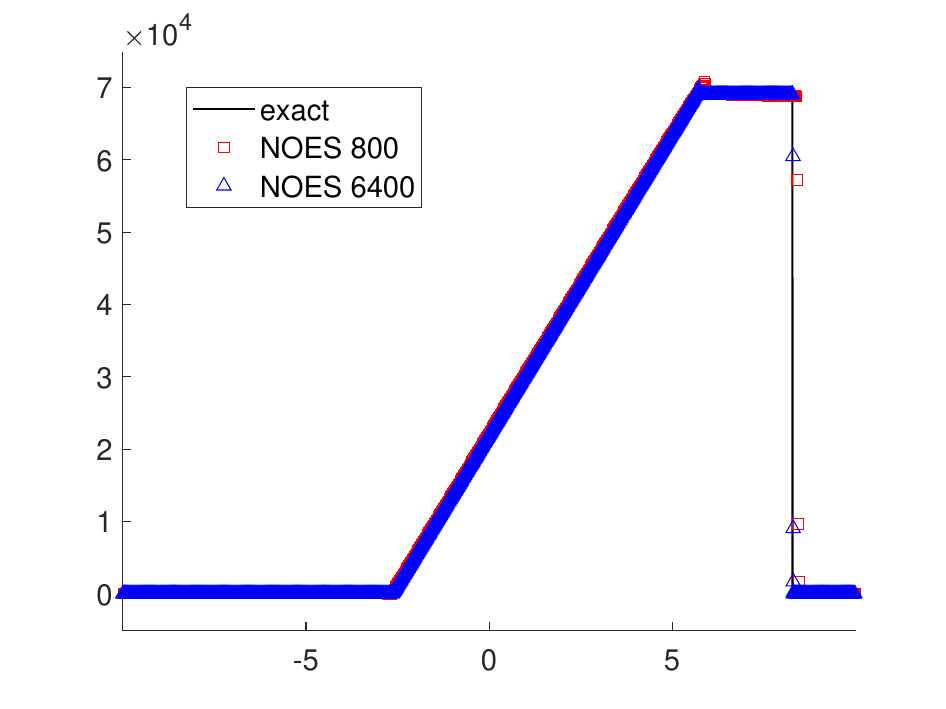}}
  \subfigure[Zoomed in (c).]{
		\includegraphics[width=0.45\linewidth]{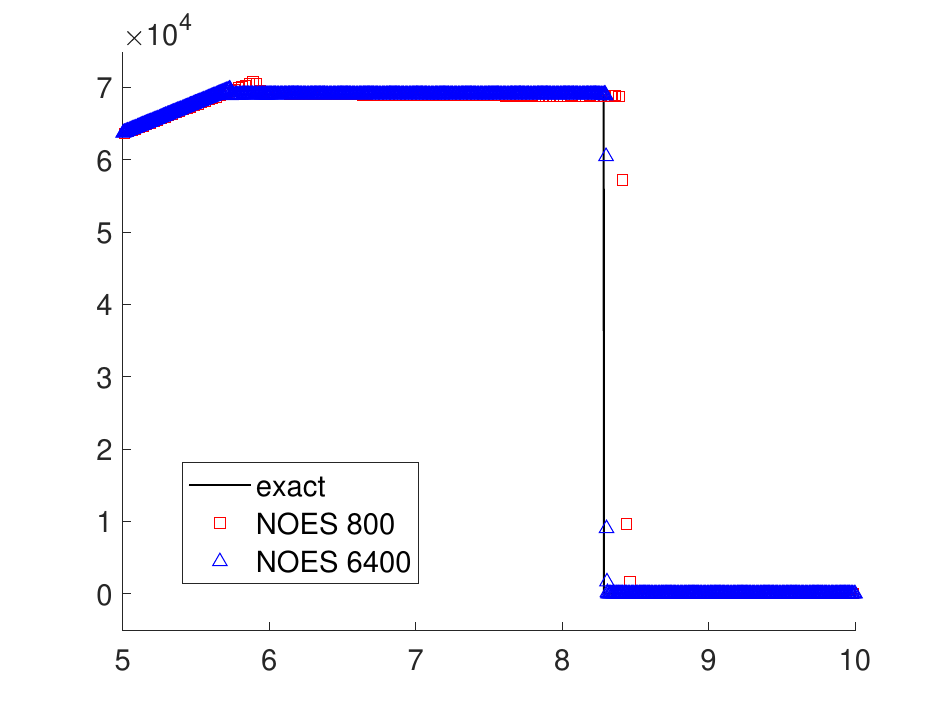}}
  \subfigure[$\lg p$.]{
		\includegraphics[width=0.45\linewidth]{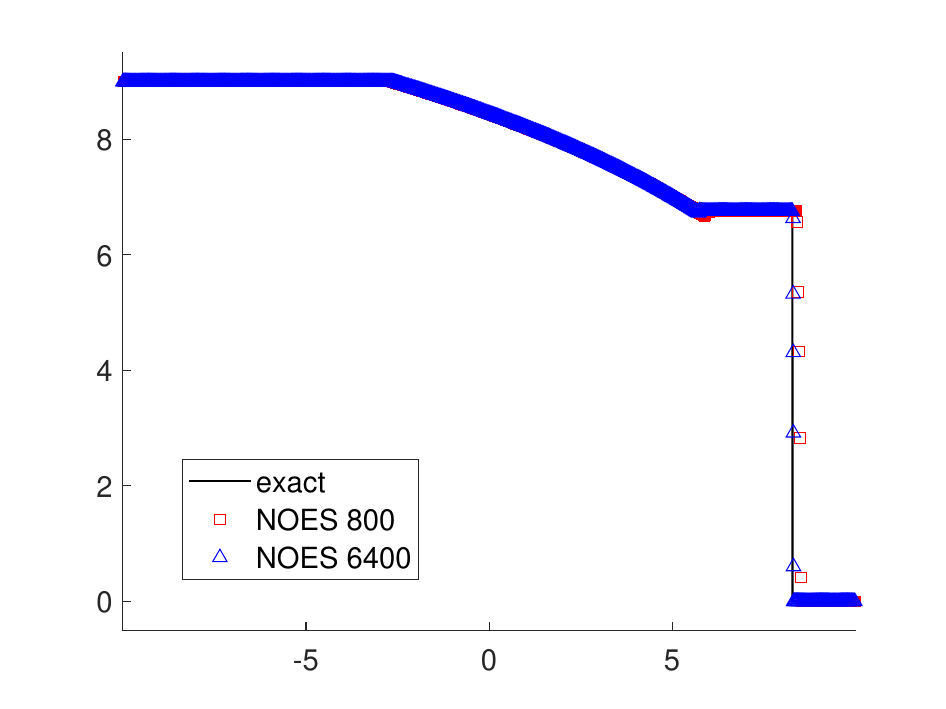}}
  \subfigure[Zoomed in (e)]{
		\includegraphics[width=0.45\linewidth]{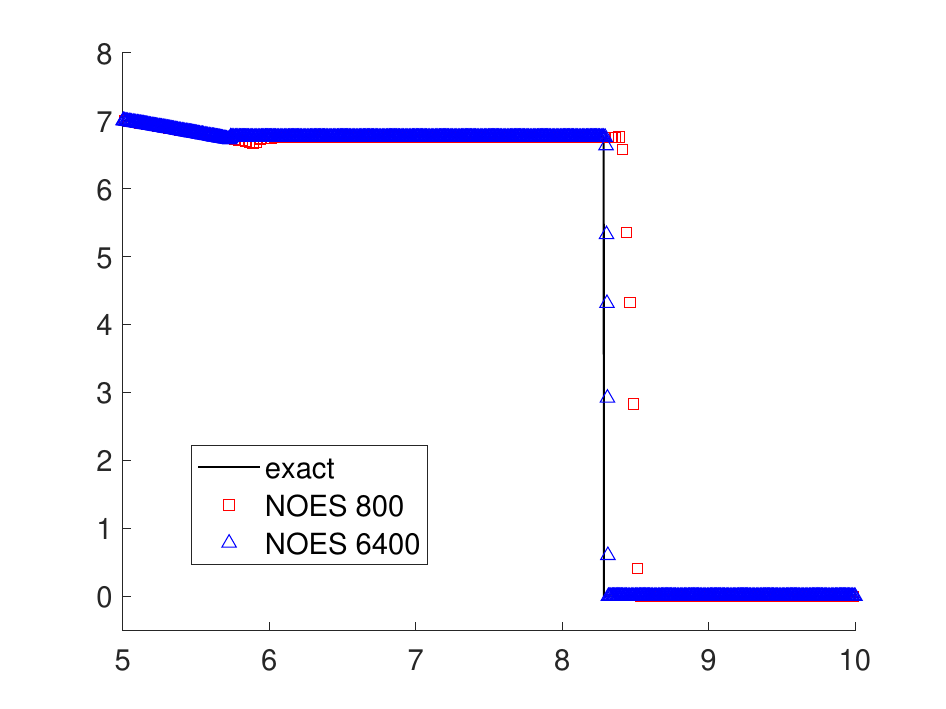}}
	\caption{Example \ref{ex:Leblanc}: Leblanc shock tube. The numerical solution at $T = 0.0001$ with $N= 800$ and $N=6400$.}
    \label{figLeblanc}
\end{figure}

\begin{Ex}
\textbf{(2D accuracy test.)}
\label{ex:vortex}
\end{Ex}

We test the accuracy of the scheme for two-dimensional Euler equation. Initially, an isentropic vortex perturbation centered at $(5,0)$ is added on the mean flow $(\rho_0,u_0,v_0,p_0)=(1,1,0,1)$:
$$
\begin{aligned}
\rho &=\left( 1-\frac{\gamma -1}{16\gamma \pi ^2}\beta ^2e^{2\left( 1-r^2 \right)} \right) ^{1/\left( \gamma -1 \right)},
\\
u&=1-\frac{\beta}{2\pi}ye^{1-r^2},\\ v&=\frac{\beta}{2\pi}\left( x-5 \right) e^{1-r^2},
\\ p&=\rho ^{\gamma},\end{aligned}
$$
where $r^2 = (x - 5)^2 + y^2$ and we take the vortex strength $\beta = 5$. This problem is essential nonlinear.
It is clear that the exact solution is just the passive convection of the vortex with the mean velocity.
In numerical simulation, we take the computational domain is $\Omega = [0,10]\times[-5,5]$ extended periodically in both directions.
In Table \ref{tab5}, we present the errors and orders of accuracy of $\rho$ at $T = 10$. At this time, the exact solution will coincide with the initial data. It can be seen that the optimal convergence rate is obtained.

\begin{table}[htb!]
        \centering
        \caption{Example \ref{ex:Burgers2D}: Two-dimensional isentropic vortex problem.  Errors and orders of density with the final time $T = 10$.}
	\setlength{\tabcolsep}{3.2mm}{
		\begin{tabular}{|c|c|cc|cc|cc|}
			\hline 
   & $N_x \times N_y$ & $L^1$ error & order & $L^2$ error & order & $L^\infty$ error & order  \\ 
    \hline  
    \multirow{5}{*}{$k=1$}
&$     64 \times     64 $  &6.43e-04 & -- &2.15e-03 & -- &2.08e-02 & -- \\  
&$    128 \times    128 $  &9.89e-05 & 2.70 &3.35e-04 &2.68 &3.58e-03 &2.54 \\
&$    256 \times    256 $  &1.45e-05 & 2.77 &4.91e-05 &2.77 &5.83e-04 &2.62 \\
&$    512 \times    512 $  &2.53e-06 & 2.52 &8.29e-06 &2.57 &1.01e-04 &2.53 \\ 
&$   1024 \times   1024 $  &5.35e-07 & 2.24 &1.73e-06 &2.26 &1.94e-05 &2.37 \\
			\hline 
   \multirow{4}{*}{$k=2$}
&$     64 \times     64 $  &1.08e-05 & -- &3.36e-05 & -- &5.36e-04 & -- \\
&$    128 \times    128 $  &1.01e-06 & 3.42 &2.96e-06 &3.50 &5.88e-05 &3.19 \\  
&$    256 \times    256 $  &1.26e-07 & 3.00 &3.75e-07 &2.98 &7.08e-06 &3.06 \\ 
&$    512 \times    512 $  &1.65e-08 & 2.93 &5.14e-08 &2.87 &8.77e-07 &3.01 \\
			\hline
   \multirow{4}{*}{$k=3$}
&$     64 \times     64 $  &6.17e-07 & -- &1.83e-06 & -- &3.45e-05 & -- \\ 
&$    128 \times    128 $  &2.61e-08 & 4.57 &8.63e-08 &4.41 &1.72e-06 &4.33 \\ 
&$    256 \times    256 $  &1.38e-09 & 4.24 &4.89e-09 &4.14 &1.10e-07 &3.97 \\  
&$    512 \times    512 $  &8.90e-11 & 3.96 &3.07e-10 &3.99 &6.70e-09 &4.04 \\
			\hline 
	\end{tabular}} 
    \label{tab5}
\end{table}

\begin{Ex}
\textbf{(2D Riemann problem.)}
\label{ex:Riemann2D}
\end{Ex}

Two-dimensional Riemann problems with different initial configurations have been
extensively employed to examine the numerical schemes for Euler equations. Here, We consider the case in \cite{kurganov2002solution}. The computational domain is $\Omega = [0,1]\times[0,1]$ with outflow boundary conditions. The initial data is given by
$$(\rho,u,v,p)=\left\{ \begin{array}{ll}   
(1.5,0,0,1.5), & (x,y)\in(0.8,1]\times(0.8,1],\\ 
(0.5323,1.206,0,0.3), & (x,y)\in[0,0.8)\times(0.8,1],\\ 
(0.138,1.206,1.206,0.029), & (x,y)\in[0,0.8)\times[0,0.8),\\ 
(0.5323,0,1.206,0.3), & (x,y)\in(0.8,1]\times[0,0.8).\\   
\end{array} \right.$$ 
In Figure \ref{figRiemann}, we present the result of density at $T = 0.8$ with $N_x=N_y=400$. 
It can be seen that the reflection shocks and contact discontinuities are captured well.

\begin{figure}[htbp]
    \centering
	\includegraphics[width=0.65\linewidth]{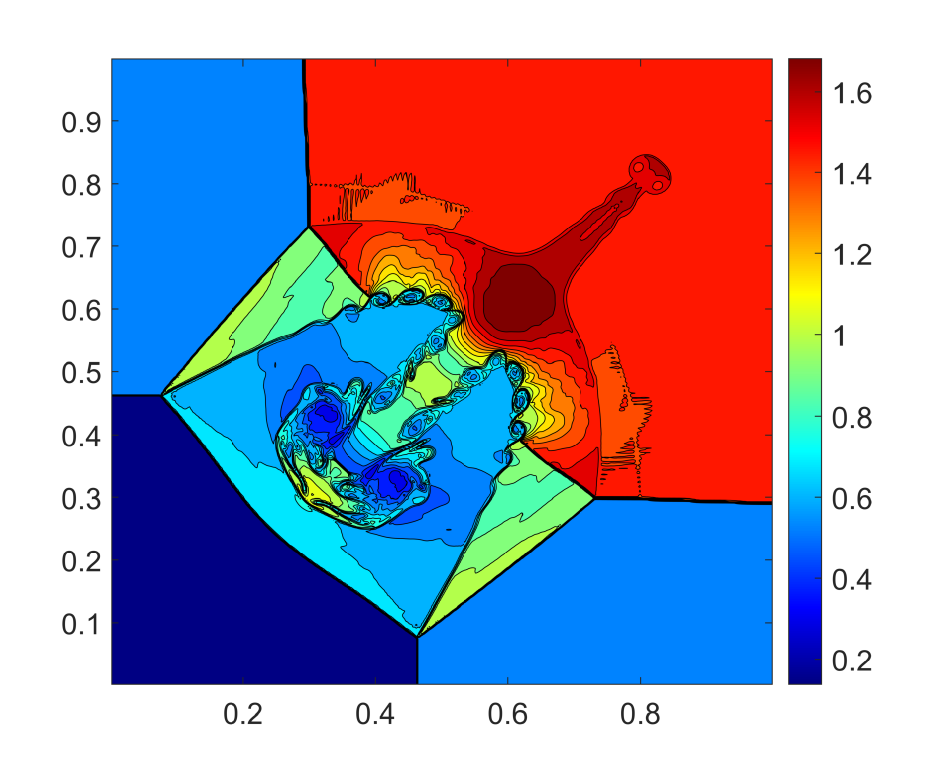}
    \caption{Example \ref{ex:Riemann2D}: Two-dimensional Riemann problem. The numerical solution of density at $T = 0.8$ with $N_x =  N_y = 400$. 30 contour lines are used. }
    \label{figRiemann}
\end{figure}

\begin{Ex}
\textbf{(Double Mach reflection.)}
\label{ex:DoubleMach}
\end{Ex}

Next, we consider the double-Mach reflection problem \cite{woodward1984numerical}, which is a benchmark test problem for the two-dimensional  Euler equation. 
The computational domain is $\Omega=[0,4]\times[0,1]$. 
The reflection wall lies at the bottom of the computational domain starting from $x=1/6$.
This problem describes a right-moving Mach 10 shock initially positioned at $x=1/6, y=0,$ makes a $60^\circ$ angle with the horizontal wall,  i.e., 
$$
\left( \rho ,u,v,p \right) =\begin{cases} 
	\left( 8,8.25\cos \left( \frac{\pi}{6} \right) ,-8.25\sin \left( \frac{\pi}{6} \right) ,116.5 \right) , & x<\frac{1}{6}+\frac{y}{\sqrt{3}}, \\
	\left( 1.4,0,0,1 \right) , & x\ge \frac{1}{6}+\frac{y}{\sqrt{3}}. \\
\end{cases}
$$
The inflow and outflow conditions are imposed on the left and the right boundary respectively. 
For the bottom boundary, the exact post-shock boundary condition is imposed for the part $0<x<1/6$, and a reflective boundary is used for the rest. 
For the top boundary, the post-shock condition is imposed for the part from $0<  x < 1/6 + (1 + 20t)/\sqrt 3$, and the pre-shock condition is used for the rest. 
In Figure \ref{figDoubleMach}, we present the result at $T = 0.2$ with $N_x\times N_y=960\times 240$ and $N_x\times N_y = 1920\times 480$. And a zoomed-in graph in provided in Figure \ref{figDoubleMach480}. 
We can see that our results agree well with the results in \cite{woodward1984numerical}.

\begin{figure}[htbp!]
	\centering
    \subfigure[$960\times 240$]{
		\includegraphics[width=0.8\linewidth]{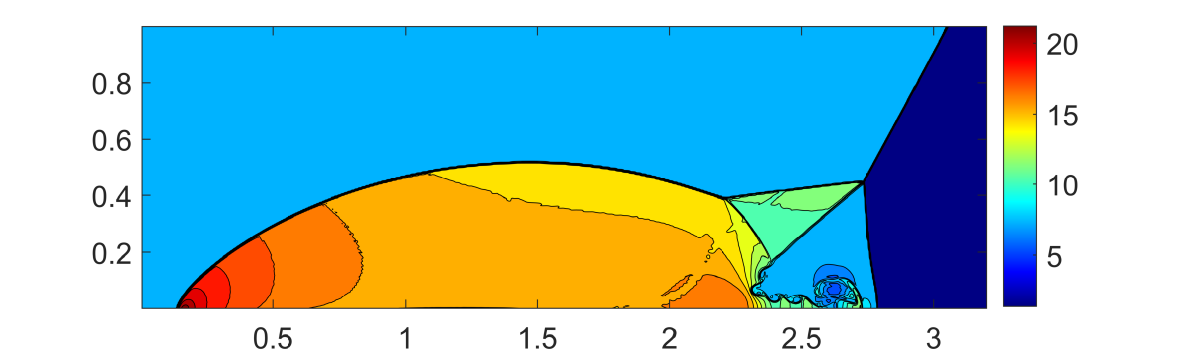}}
	\subfigure[$1920\times 480$]{
		\includegraphics[width=0.8\linewidth]{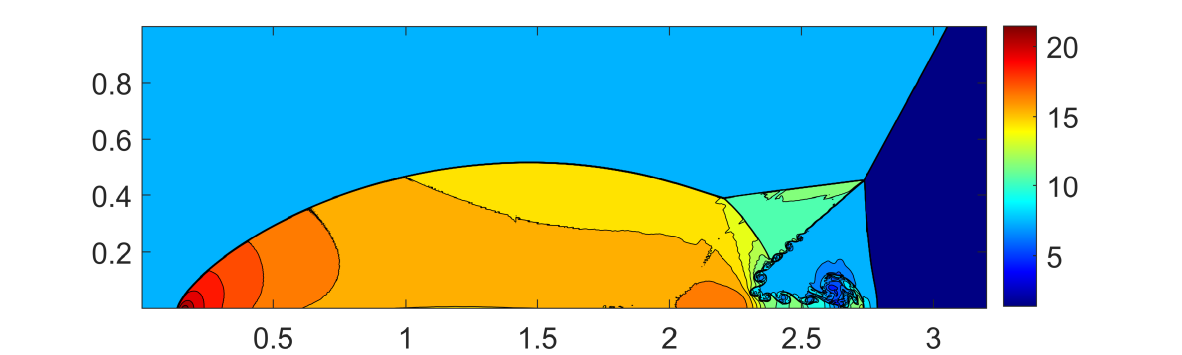}}
	\caption{Example \ref{ex:DoubleMach}: Double Mach reflection. The numerical solution of density at $T = 0.2$. 30 contour lines are used. }
    \label{figDoubleMach}
\end{figure}

\begin{figure}[htbp!]
	\centering
    \subfigure[$960\times 240$]{
		\includegraphics[width=0.45\linewidth]{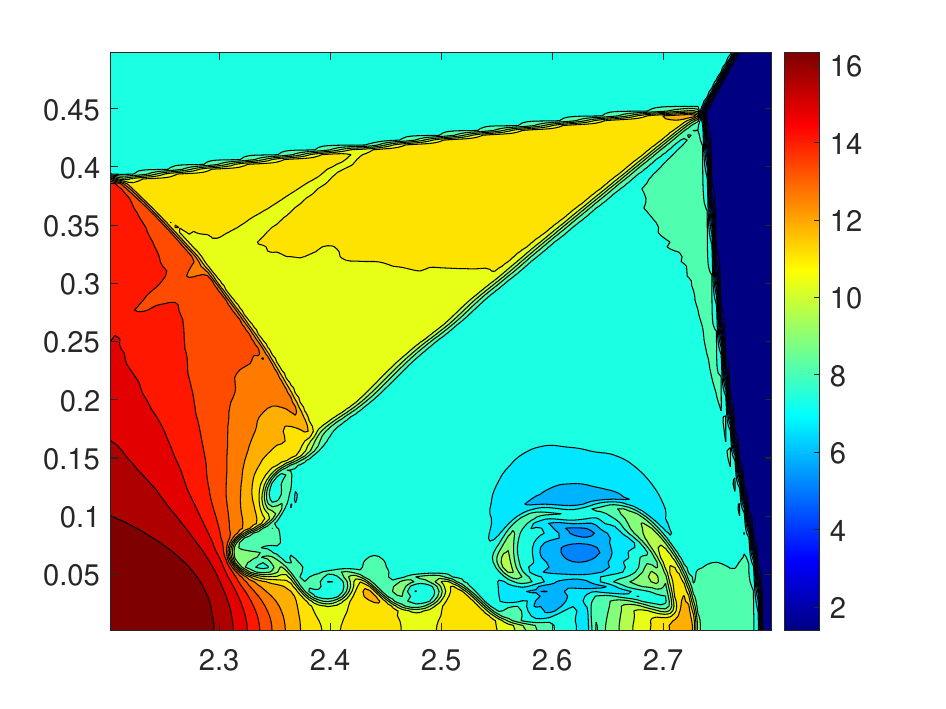}}
	\subfigure[$1920\times 480$]{
		\includegraphics[width=0.45\linewidth]{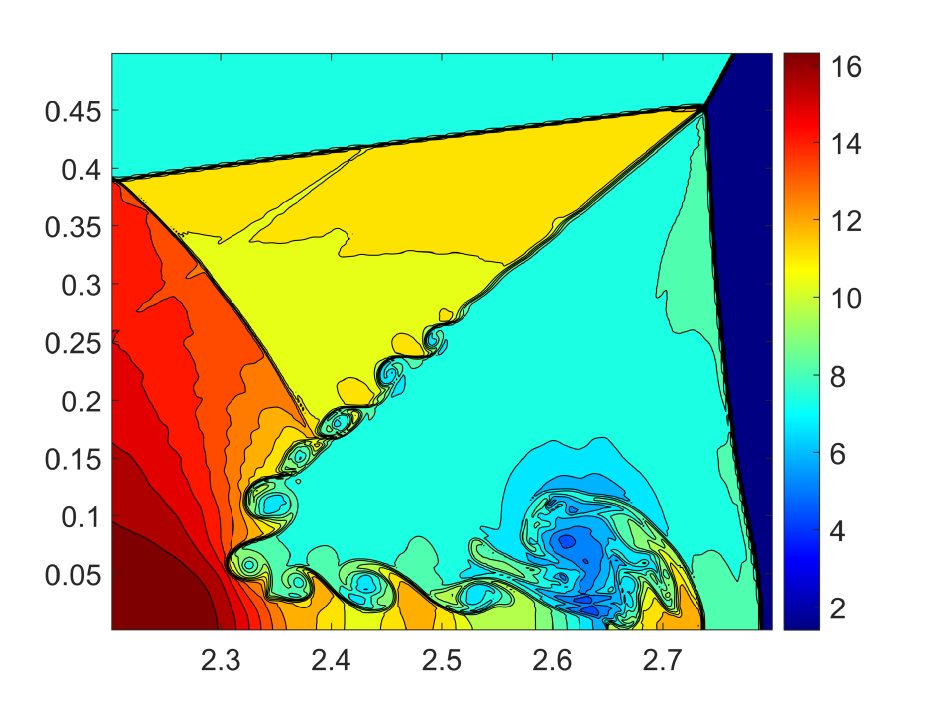}}
	\caption{Example \ref{ex:DoubleMach}: Double Mach reflection. Zoomed-in figure. The numerical solution of density near the Mach stem at $T = 0.2$. 30 contour lines are used. }
    \label{figDoubleMach480}
\end{figure}

\begin{Ex}
\textbf{(Mach 2000 jet.)}
\label{ex:jet}
\end{Ex}

Finally, we consider the high Mach number astrophysical jet problem \cite{zhang2010positivity}.  We note that $\gamma = 5/3$ in this example. 
In this example, the code could easily produce the negative pressure
and density, resulting in blow up easily during numerical computation.

The computational domain is $\Omega = [0,1]\times[-0.25,0.25]$, initially full of the ambient gas with
$$(\rho,u,v,p) = (0.5,0,0,0.4127).$$
A Mach 2000 jet state $(\rho,u,v,p) = (5,800,0,0.4127)$ is injected from the left boundary in the range $-0.05<y<0.05$. Outflow conditions are applied on all remaining boundaries. This example requires the PP limiter. 

In Figure \ref{figjet}, we present the density, pressure and temperature $p/\rho$ at time $T = 0.001$ with $320\times 160$ meshes. 
It is observed that our scheme can capture the feature without any occurrence of instability, indicating the efficiency of our proposed schemes.

\begin{figure}[htbp!]
	\centering
    \subfigure[Density.]{
		\includegraphics[width=0.45\linewidth]{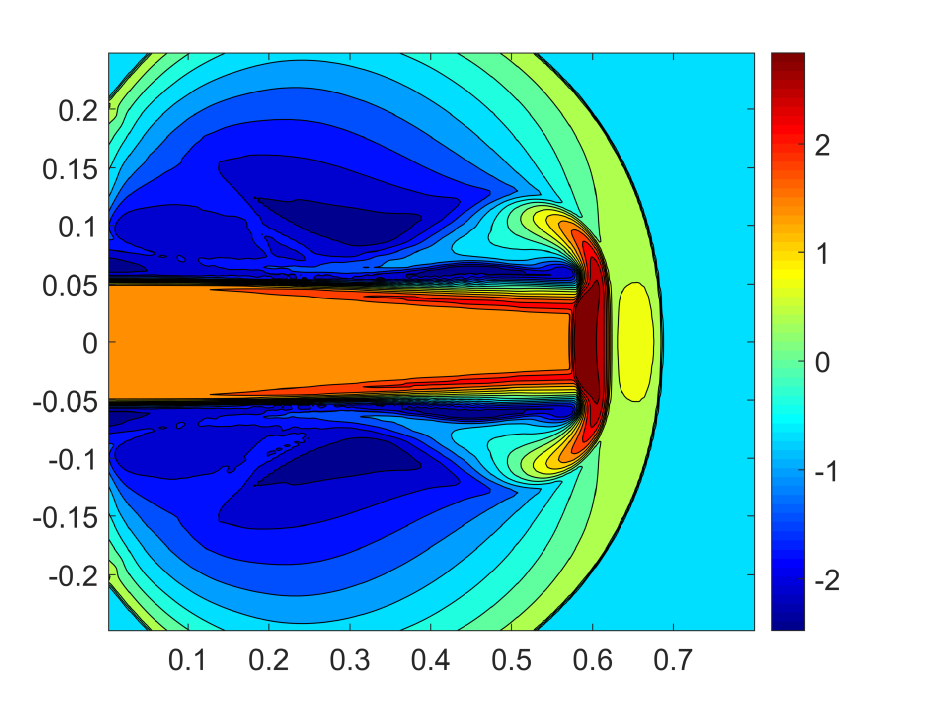}}
    \subfigure[Pressure.]{
		\includegraphics[width=0.45\linewidth]{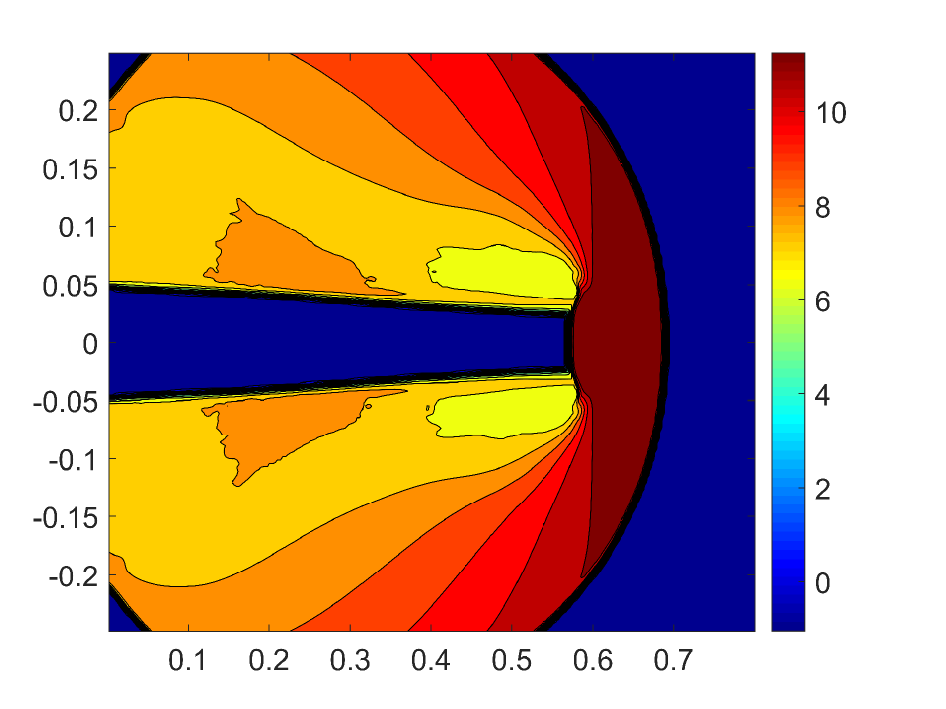}}
	\subfigure[Temperature.]{
		\includegraphics[width=0.45\linewidth]{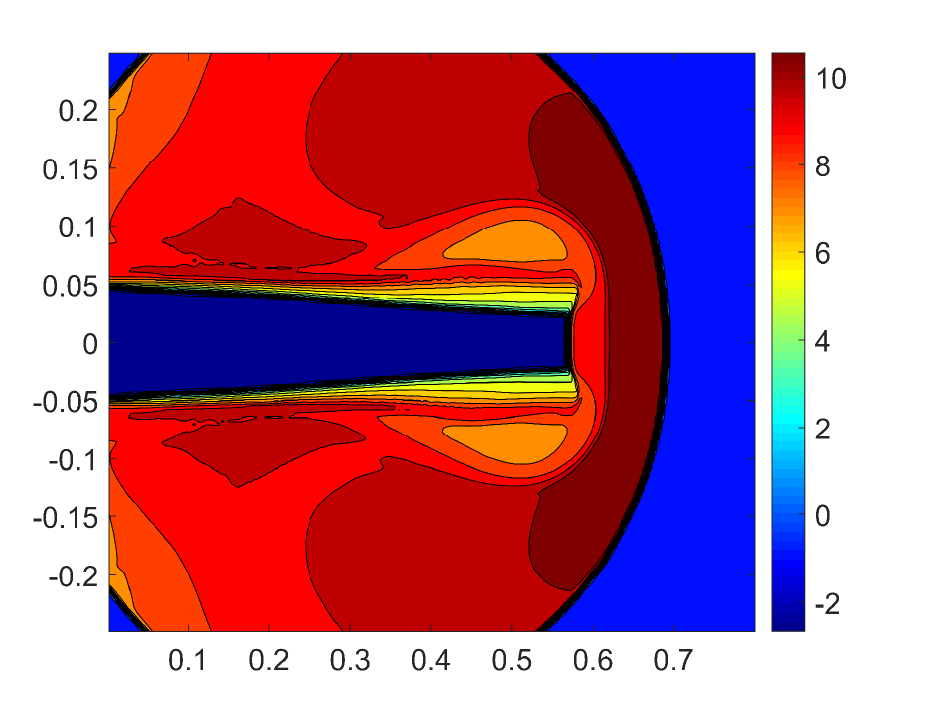}}
	\caption{Example \ref{ex:jet}: High Mach number jet problem. The numerical solution  at $T = 0.001$ with $N_x \times  N_y = 320\times 160$. Scales are logarithmic. 12 contour lines are used.}
    \label{figjet}
\end{figure}

\section{Concluding remarks}

In this paper, we present a class of non-oscillatory entropy-stable DG methods for solving hyperbolic conservation laws. By incorporating a specific form of local artificial viscosity term into the original DG scheme on each cell and imposing an interpolation condition, the semi-discrete scheme ensures strict non-increase of the total entropy with a careful design of the viscosity coefficient. 
A damping coefficient is taken into account to control spurious oscillation near shock waves. 
Moreover, this damping coefficient is also employed to constrain the viscosity parameter in cases where the entropy coefficient’s denominator approaches zero.
We prove the optimal accuracy for scalar advection equations with quadratic entropy, and advection equations with other entropy functions under certain assumptions. 
It is noted that the added viscosity term may induce stiffness and impose strict constraints on time step sizes. To address this, we employ the integrating factor Runge-Kutta method, enabling explicit computation under the standard CFL condition. Due to the local structure of our proposed scheme, the algorithm can be implemented easily. 
For extreme problems where negative pressures might occur, we demonstrate that our scheme is compatible with the positivity-preserving limiter. 
Extensive numerical examples confirm that our scheme ensures entropy stability, suppresses oscillations, and attains optimal accuracy. 
Future work involves extending the proposed method to unstructured meshes and exploring its application to other equations.

\section*{Acknowledgements}

The authors would like to express their gratitude to Professor Yang Yang for his valuable discussions.

\end{document}